\documentclass[onefignum,onetabnum,pagebackref]{siamart220329}

\usepackage{lipsum}
\usepackage{amsfonts}
\usepackage{graphicx}
\usepackage{epstopdf}
\ifpdf
  \DeclareGraphicsExtensions{.eps,.pdf,.png,.jpg}
\else
  \DeclareGraphicsExtensions{.eps}
\fi
\usepackage[T1]{fontenc}

\usepackage{enumitem}
\setlist[enumerate]{leftmargin=.5in}
\setlist[itemize]{leftmargin=.5in}

\usepackage{amsmath, amssymb, mathtools, stmaryrd}
\usepackage{amsfonts,mathrsfs, mleftright,booktabs,soul}
\usepackage{subcaption}
\usepackage{xspace}
\usepackage{algorithm, algpseudocode, algorithmicx}
\usepackage{backref}

\newsiamthm{fact}{Fact}
\newsiamthm{question}{Question}

\headers{XTrace for Stochastic Trace Estimation}{E. N. Epperly, J. A. Tropp, and R. J. Webber}

\title{XT{\bfseries\scshape\MakeLowercase{race}}: Making the most of every sample \\ in stochastic trace estimation\thanks{This is a preprint of \emph{\XTrace: Making the most of every sample in stochastic trace estimation} (\url{https://doi.org/10.1137/23M1548323}), which appeared in the SIAM Journal on Matrix Analysis and Applications on January 3, 2024.
\funding{ENE acknowledges support from the  U.S. Department of Energy, Office of Science, Office of Advanced Scientific Computing Research,  Department of Energy Computational Science Graduate Fellowship under Award Number DE-SC0021110.
JAT and RJW acknowledge support from the Office of Naval Research through BRC Award N00014-18-1-2363
and from the National Science Foundation through FRG Award 1952777.}}}

\author{Ethan N. Epperly\thanks{Division of Computing and Mathematical Sciences, California Institute of Technology, Pasadena, CA 91125 USA (\email{eepperly@caltech.edu}, \email{jtropp@caltech.edu}, \email{rwebber@caltech.edu}).}
\and Joel A. Tropp\footnotemark[2] \and Robert J. Webber\footnotemark[2]}

\usepackage{amsopn}

\newcommand{\real}{\mathbb{R}}

\renewcommand{\left}{\mleft}
  \renewcommand{\right}{\mright}

\DeclareMathOperator{\tr}{tr}
\DeclareMathOperator{\diag}{diag}
\newcommand{\mat}[1]{\boldsymbol{#1}}
\renewcommand{\vec}[1]{\boldsymbol{#1}}
\newcommand{\lowrank}[1]{\left\llbracket #1 \right\rrbracket}
\newcommand{\norm}[1]{\left\| #1 \right\|}

\DeclareMathOperator{\rank}{rank}
\DeclareMathOperator{\orth}{orth}

\newcommand{\expmat}[1]{\begin{bmatrix} #1 \end{bmatrix}}
\newcommand{\twobytwo}[4]{\expmat{#1 & #2 \\ #3 & #4}}

\newcommand{\Id}{\mathbf{I}}
\newcommand{\evec}{\mathbf{e}}

\DeclareMathOperator{\Var}{Var}
\DeclareMathOperator{\Cov}{Cov}
\DeclareMathOperator{\Cor}{Cor}
\DeclareMathOperator{\expect}{\mathbb{E}}

\DeclareMathOperator{\Unif}{\textsc{Unif}}

\newcommand{\order}{\mathcal{O}}

\newcommand{\set}[1]{\mathsf{#1}}
\newcommand{\e}{\mathrm{e}}

\renewcommand{\hat}[1]{\widehat{#1}}

\usepackage{listings}
\usepackage{color} 
\definecolor{mygreen}{RGB}{28,172,0} 
\definecolor{mylilas}{RGB}{170,55,241}
\newcommand{\XTrace}{\textsc{XTrace}\xspace}
\newcommand{\NysTrace}{\textsc{XNysTrace}\xspace}

\newcommand{\XDiag}{\textsc{XDiag}\xspace}

\newcommand{\HutchPP}{\textsc{Hutch}\textnormal{\texttt{++}}\xspace}
\newcommand{\DiagPP}{\textsc{Diag}\textnormal{\texttt{++}}\xspace}
\newcommand{\Hutch}{\textsc{Hutch}\xspace}
\newcommand{\LRA}{\textsc{LRA}\xspace}
\newcommand{\nys}[2]{#1\mleft\langle #2 \mright\rangle}
\crefname{lstlisting}{Program}{Programs}
\Crefname{lstlisting}{Program}{Programs}
\newcommand{\err}{\mathrm{err}}

\renewcommand{\epsilon}{\varepsilon}
\newcommand{\NysPP}{\textsc{Nystr\"om}\texttt{++}\xspace}

\usepackage{xcolor}

\ifpdf
\hypersetup{
	pdftitle={XTrace for Stochastic Trace Estimation},
	pdfauthor={E. N. Epperly, J. A. Tropp, and R. J. Webber}
}
\fi

\begin{document}
	\counterwithin{lstlisting}{section}
	
	\maketitle
	
	\begin{abstract}
		The implicit trace estimation problem asks for an approximation of
		the trace of a square matrix, accessed via matrix--vector products (matvecs).
		This paper designs new randomized algorithms, \textsc{XTrace}\xspace and \NysTrace,
		for the trace estimation problem by exploiting both variance reduction
		and the exchangeability principle.  For a fixed budget of matvecs,
		numerical experiments show that the new methods can achieve errors
		that are orders of magnitude smaller than existing algorithms,
		such as the Girard--Hutchinson estimator or the \HutchPP estimator.
		A theoretical analysis confirms the benefits by offering a precise
		description of the performance of these algorithms as a function of the spectrum of the input matrix.
		The paper also develops an exchangeable estimator, \XDiag, for approximating the diagonal
		of a square matrix using matvecs.
	\end{abstract}
	
	
	
	\section{Introduction} 
	
	Over the past three decades, researchers have developed \emph{randomized algorithms} for linear algebra problems such as trace estimation \cite{Gir89,Hut89,MMMW21}, low-rank approximation \cite{HMT11}, and over-determined least squares \cite{AMT10,RT08}.
	Many of these algorithms collect information by judicious random sampling of the problem data.  As a consequence, we can design better algorithms using techniques from the theory of statistical estimation, such as variance reduction and the exchangeability principle.  This paper explores how the exchangeability principle leads to faster randomized algorithms for trace estimation.
	
	Suppose that we wish to compute a quantity $Q(\mat{A})$ associated with a matrix $\mat{A}$.  A typical randomized algorithm might proceed as follows.
	\begin{enumerate}
		\item \textbf{Collect information} about the matrix $\mat{A}$ by computing matrix--vector products $\mat{A}\vec{\omega}_1,\ldots,\mat{A}\vec{\omega}_k$ with random test vectors $\vec{\omega}_1,\ldots,\vec{\omega}_k$.
		\item \textbf{Form an estimate} of $Q(\mat{A})$ from the samples $\mat{A}\vec{\omega}_1,\ldots,\mat{A}\vec{\omega}_k$.
	\end{enumerate}
	The question arises: \emph{Given the data $\mat{A}\vec{\omega}_1,\ldots,\mat{A}\vec{\omega}_k$, what is an optimal estimator for $Q(\mat{A})$?}
	One property an optimal estimator must obey is the exchangeability principle:
	
	\begin{center} \vspace{0.5pc}
		\fbox{ \begin{minipage}{0.9\textwidth}
				\textbf{Exchangeability principle:} If the test vectors $\vec{\omega}_1,\ldots,\vec{\omega}_k$ are exchangeable, the
				minimum-variance unbiased estimator for $Q(\mat{A})$ is always a symmetric function of $\vec{\omega}_1,\ldots,\vec{\omega}_k$.
			\end{minipage}
		} \vspace{0.5pc}
	\end{center}
	``Exchangeability'' means that the family $(\vec{\omega}_1,\ldots,\vec{\omega}_k)$ has the same distribution as 
	the permuted family $(\vec{\omega}_{\sigma(1)},\ldots,\vec{\omega}_{\sigma(k)})$ for every permutation $\sigma$ in the symmetric group $\set{S}_k$.  In particular, an independent and identically distributed (iid) family is exchangeable.
	
	The implication of the exchangeability principle is that our estimators should be symmetric functions of the samples, whenever possible.  This idea is attributed to Halmos~\cite{Hal46}, and it plays a central role in the theory of U-statistics~\cite{KB94}. 
	
	This paper will demonstrate that the exchangeability principle can lead to new randomized algorithms for linear algebra problems.
	As a case study, we will explore the problem of \emph{implicit trace estimation}:
	\begin{center} \vspace{0.5pc}
		\fbox{ \begin{minipage}{0.9\textwidth}
				\textbf{Implicit trace estimation problem:} Given access to a square matrix $\mat{A}$ via the matrix--vector product (matvec) operation $\vec{\omega} \mapsto \mat{A}\vec{\omega}$, estimate the trace of $\mat{A}$.
			\end{minipage}
		} \vspace{0.5pc}
	\end{center}
	Trace estimation plays a role in a wide range of areas, including computational statistics,
	statistical mechanics, and network analysis.  See the survey~\cite{US17} for more applications.
	
	As we will see, it is natural to design randomized algorithms for trace estimation that use matvecs
	between the input matrix and random test vectors.  At present, the state-of-the-art trace estimators
	do not satisfy the exchangeability principle.  By pursuing this insight, we will develop better
	trace estimators.  Given a fixed budget of matvecs, the new algorithms can reduce the variance of the trace estimate
	by several orders of magnitude.  This case study highlights the importance of enforcing
	exchangeability in the design of randomized algorithms.
	
	\subsection{Stochastic trace estimators}
	\label{sec:existing_ests}
	
	In this section, we outline the classic approach to randomized trace estimation
	based on Monte Carlo approximation.  Then we introduce a more modern approach
	that incorporates a variance reduction strategy.
	
	\subsubsection{The Girard--Hutchinson estimator}
	
	The first randomized algorithm for trace estimation was proposed by Girard \cite{Gir89}
	and extended by Hutchinson \cite{Hut89}.
	
	Let $\mat{A} \in \real^{N \times N}$ be a square input matrix.  Consider an isotropic random vector $\vec{\omega} \in \real^N$:
	\begin{equation} \label{eq:isotropic}
		\expect[ \vec{\omega}^{\vphantom{*}}\vec{\omega}^* ] = \Id.
	\end{equation}
	For example, we may take a random sign vector $\vec{\omega} \sim \textsc{uniform}\{\pm 1\}^N$.
	By isotropy,
	$$
	\expect[ \vec{\omega}^* (\mat{A} \vec{\omega}) ] = \tr \mat{A}.
	$$
	The symbol ${}^*$ denotes the transpose.  This relation suggests a Monte Carlo method.
	
	Accordingly, the Girard--Hutchinson trace estimator takes the form
	\begin{equation} \label{eq:girard_hutchinson}
		\hat{\tr}_{\rm GH} \coloneqq \frac{1}{m} \sum_{i=1}^m \vec{\omega}_i^* (\mat{A}\vec{\omega}_i^{\vphantom{*}})
		\quad\text{where the $\vec{\omega}_i$ are iid copies of $\vec{\omega}$.}
	\end{equation}
	This estimator is exchangeable, and it is unbiased: $\expect
	\big[ \hat{\tr}_{\rm GH} \big] = \tr \mat{A}$.  We can measure the quality of the estimator using the variance, $\Var\big[ \hat{\tr}_{\rm GH} \big]$.  The variance depends on the matrix $\mat{A}$ and the distribution of $\vec{\omega}$,
	but it converges to zero at the Monte Carlo rate $\Theta(m^{-1})$ as we increase the number $m$ of samples.
	See the survey~\cite[\S4]{MT20a} for more discussion.
	
	\subsubsection{The \HutchPP estimator}
	
	To improve on the Girard--Hutchinson estimator, several papers~\cite{GSO17,Lin17,MMMW21,SAI17}
	have advocated variance reduction techniques.  The key idea is to form
	a low-rank approximation of the input matrix.  We can compute the trace of the approximation
	exactly (as a control variate), so we only need to estimate the trace of the residual.
	This approach can attain lower variance than the Monte Carlo method.
	
	The \HutchPP estimator of Meyer, Musco, Musco, and Woodruff~\cite{MMMW21}
	crystallizes the variance reduction strategy.  Let $\mat{A} \in \real^{N \times N}$ be a square input matrix.
	Given a fixed budget of $m$ matvecs, with $m$ divisible by $3$, \HutchPP proceeds as follows:
	\begin{enumerate}
		\item \textit{Sample} iid isotropic vectors $\vec{\omega}_1 ,\ldots, \vec{\omega}_{2m/3}\in\real^N$ as in~\cref{eq:isotropic}.
		\item \textit{Sketch} $\mat{Y} = \mat{A}\begin{bmatrix} \vec{\omega}_{m/3+1} & \vec{\omega}_{m/3+2} & \cdots & \vec{\omega}_{2m/3}\end{bmatrix}$
		\item \textit{Orthonormalize} $\mat{Q} = \orth(\mat{Y})$.
		\item \textit{Output} the estimate
		\begin{equation} \label{eq:hutch++}
			\hat{\tr}_{\rm H++} \coloneqq \tr(\mat{Q}^* (\mat{A}\mat{Q})) +  \frac{1}{m/3} \sum_{i=1}^{m/3} \vec{\omega}_i^*(\Id-\mat{Q}\mat{Q}^*) \big(\mat{A}(\Id-\mat{Q}\mat{Q}^*)\vec{\omega}_i^{\vphantom{*}}\big).
		\end{equation}
	\end{enumerate}
	See \cref{alg:hutchpp} for efficient \HutchPP pseudocode.
	
	\begin{algorithm}[t]
		\caption{\HutchPP~\cite{MMMW21} \label{alg:hutchpp}}
		\begin{algorithmic}[1]
			\Require Matrix $\mat{A}\in\real^{N\times N}$ and number $m$ of matvecs, where $m$ is divisible by $3$
			\Ensure Trace estimate $\hat{\tr} \approx \tr \mat{A}$
			\State Draw iid isotropic $\vec{\omega}_1,\ldots,\vec{\omega}_{2m/3} \in \real^N$ \Comment{For example, $ \vec{\omega}_i \sim \textsc{uniform} \{\pm 1\}^{N}$}
			\State $\mat{Y} \leftarrow \mat{A}\begin{bmatrix} \vec{\omega}_{m/3 + 1} & \cdots & \vec{\omega}_{2m/3} \end{bmatrix}$
			\Comment{Use matvecs}
			\State $\mat{Q} \leftarrow \orth(\mat{Y})$
			\State $\mat{G} \leftarrow \begin{bmatrix} \vec{\omega}_1 & \cdots & \vec{\omega}_{m/3}\end{bmatrix} - \mat{Q}\mat{Q}^*\begin{bmatrix} \vec{\omega}_1 & \cdots & \vec{\omega}_{m/3}\end{bmatrix}$
			\State $\hat{\tr} \leftarrow \tr\big(\mat{Q}^* (\mat{A}\mat{Q}^{\vphantom{*}}) \big) + (m/3)^{-1}\tr(\mat{G}^* (\mat{A}\mat{G}))$
			\Comment{Use matvecs}
		\end{algorithmic}
	\end{algorithm}
	
	To illustrate how \HutchPP takes advantage of low-rank approximation,
	we first observe that $\mat{\hat{A}} = \mat{Q}\mat{Q}^*\mat{A}$ is a low-rank approximation of the matrix $\mat{A}$.
	Indeed, the matrix $\mat{\hat{A}}$ coincides with the randomized SVD \cite{HMT11} formed from the test matrix $\begin{bmatrix} \vec{\omega}_{m/3 + 1} & \cdots & \vec{\omega}_{2m/3} \end{bmatrix}$.
	\HutchPP computes the trace of the low-rank approximation:
	\begin{equation*}
		\tr \mat{\hat{A}} = \tr\bigl(\mat{Q}\mat{Q}^*\mat{A}\bigr) = \tr\bigl(\mat{Q}^* \mat{A} \mat{Q} \bigr).
	\end{equation*}
	Afterward, \HutchPP applies the Girard--Hutchinson estimator to estimate the trace of the residual 
	\begin{equation*}
		\tr\bigl(\mat{A}-\mat{\hat{A}}\bigr) = \tr\bigl((\Id - \mat{Q}\mat{Q}^*)\mat{A}\bigr)= \tr\bigl((\Id - \mat{Q}\mat{Q}^*)\mat{A}(\Id - \mat{Q}\mat{Q}^*)\bigr).
	\end{equation*}
	Like the Girard--Hutchinson trace estimator, the \HutchPP estimator is unbiased.
	In contrast to the $\Theta(m^{-1})$ variance of Girard--Hutchinson,
	the variance of \HutchPP is no greater than $\order(m^{-2})$.
	In practice, the reduction in variance is conspicuous.
	However, the \HutchPP estimator violates the exchangeability
	principle, so we recognize an opportunity to design a better algorithm.
	
	\subsection{New exchangeable trace estimators}
	\label{sec:ests}
	
	The \HutchPP estimator is not exchangeable because it uses some test vectors to perform low-rank approximation, while it uses other test vectors to estimate the trace of the residual.
    Although it might seem natural to symmetrize \HutchPP over all splits of the test vectors, this approach is computationally infeasible due the combinatorial explosion in the number ${2m/3 \choose m/3} \ge 2^{m/3}$ of assignments of $2m/3$ test vectors to two groups of $m/3$.

    To circumvent this obstacle, we develop a new family of exchangeable trace estimators that use an unbalanced splitting of $m/2-1$ test vectors for low-rank approximation and just $1$ test vector for residual trace estimation.
    By symmetrizing this unbalanced estimator, we effectively use \emph{all} of the test vectors for low-rank
	approximation and for estimating the trace of the residual.
	To make this efficient, we use a leave-one-out technique that can be implemented at the same computational
	cost as \HutchPP.
	This innovation can reduce the variance
	by several orders of magnitude.
	
	\subsubsection{The \textsc{XTrace}\xspace estimator}
	\label{sec:xtrace-intro}
	
	Our first method, called \textsc{XTrace}\xspace, is an exchangeable trace estimator designed
	for general square matrices.  It computes a family of variance-reduced trace estimators.
	Each estimator uses all but one test vector to form a low-rank approximation,
	and it uses the remaining test vector to estimate the trace of the residual.
	\textsc{XTrace}\xspace then averages the basic estimators together to obtain an exchangeable trace estimator.
	
	Let us give a more detailed description.
	Fix a square input matrix $\mat{A} \in \real^{N \times N}$.  The parameter $m$ is the number of matvecs,
	where $m$ is an even number.  Draw an iid family $\vec{\omega}_1, \dots, \vec{\omega}_{m/2} \in \real^N$
	of isotropic test vectors, and define the test matrix
	$$
	\mat{\Omega} = \begin{bmatrix} \vec{\omega}_1 & \vec{\omega}_2 & \vec{\omega}_3 & \dots & \vec{\omega}_{m/2}
	\end{bmatrix}.
	$$
	Construct the orthonormal matrices
	\begin{equation} \label{eq:Q_i}
		\mat{Q}_{(i)} = \orth( \mat{A} \mat{\Omega}_{-i} )
		\quad\text{for each $i = 1, \dots, m/2$},
	\end{equation}
	where $\mat{\Omega}_{-i}$ is the test matrix with the $i$th column removed.
	Compute the basic trace estimators
	\begin{equation} \label{eq:trX-i}
		\hat{\tr}_{i} \coloneqq \tr\big(\mat{Q}_{(i)}^* \big(\mat{A}\mat{Q}_{(i)}^{\vphantom{*}}\big)\big) + \vec{\omega}_i^*\big(\Id-\mat{Q}_{(i)}^{\vphantom{*}}\mat{Q}_{(i)}^*\big) \big(\mat{A}\big(\Id-\mat{Q}_{(i)}^{\vphantom{*}}\mat{Q}_{(i)}^*\big)\vec{\omega}_i^{\vphantom{*}}\big)
	\end{equation}
    for $i = 1, \dots, m/2$.
	The \textsc{XTrace}\xspace estimator averages these basic estimators:
	\begin{equation} \label{eq:xtrace}
		\hat{\tr}_{\rm X} \coloneqq \frac{1}{m/2} \sum_{i=1}^{m/2} \hat{\tr}_i.
	\end{equation}
	The \textsc{XTrace}\xspace method gives an unbiased, exchangeable estimate for the trace. \Cref{thm:main}
	provides a detailed a priori bound for the variance.  We can also obtain an a posteriori estimate for the error using the formula
	$$
	\hat{\err}_{\rm X}^2 \coloneqq \frac{1}{(m/2)(m/2 - 1)} \sum_{i=1}^{m/2} (\hat{\tr}_i - \hat{\tr}_{\rm X})^2.
	$$
	\Cref{sec:error} contains further discussion of the error estimate.
	See \cref{alg:xtrace_naive} for a na\"ive implementation of \textsc{XTrace}\xspace.  
 
	While it may not seem obvious from \cref{eq:Q_i,eq:trX-i,eq:xtrace}, \textsc{XTrace}\xspace requires exactly $m$ matvecs with the input matrix $\mat{A}$.
    Indeed, as we detail in \cref{sec:xtrace}, all the information needed to form the \textsc{XTrace}\xspace estimator can be collected in two batches of matvecs.
    First, we compute $\mat{Y} \coloneqq  \mat{A}\mat{\Omega}$ where $\mat{\Omega} \coloneqq \begin{bmatrix} \vec{\omega}_1 & \cdots & \vec{\omega}_{m/2}\end{bmatrix}$.
    Then, we orthonormalize the matvecs to obtain $\mat{Q} \coloneqq \orth(\mat{Y})$.
    Finally, we use our second round of matvecs to compute $\mat{A}\mat{Q}$.
    With careful attention to the
	linear algebra, we can use $\mat{A} \mat{\Omega}$ and $\mat{A}\mat{Q}$ to form the \textsc{XTrace}\xspace estimator at the same $\order(m^2N)$ cost as computing a single estimator~\cref{eq:trX-i}; see \cref{sec:xtrace} and \cref{sec:xtrace_derivation} for details.
	
	\begin{algorithm}[t]
		\caption{\textsc{XTrace}\xspace: Na{\"i}ve implementation \label{alg:xtrace_naive}}
		\begin{algorithmic}[1]
			\Require Matrix $\mat{A} \in \real^{N \times N}$ and number $m$ of matvecs, where $m$ is even
			\Ensure Trace estimate $\hat{\tr} \approx \tr \mat{A}$ and error estimate $\hat{\err} \approx \vert\hat{\tr} - \tr \mat{A}\vert$
			\State Draw iid isotropic $\vec{\omega}_1,\ldots,\vec{\omega}_{m/2} \in \real^N$ \Comment{See \cref{sec:distribution}}
			\State $\mat{Y} \leftarrow \mat{A}\begin{bmatrix} \vec{\omega}_1 & \cdots & \vec{\omega}_{m/2}\end{bmatrix}$
			\Comment{Use matvecs}
			\For{$i = 1,2,3,\ldots,m/2$}
			\State $\mat{Q}_{(i)} \leftarrow \orth(\mat{Y}_{-i})$
			\Comment{Remove $i$th column of $\mat{Y}$}
			\State $\hat{\tr}_i \leftarrow \tr\big(\mat{Q}_{(i)}^*(\mat{A}\mat{Q}_{(i)}^{\vphantom{*}})\big) + \vec{\omega}_i^*\big(\Id-\mat{Q}_{(i)}^{\vphantom{*}}\mat{Q}_{(i)}^*\big) \big( \mat{A}\big(\Id-\mat{Q}_{(i)}^{\vphantom{*}}\mat{Q}_{(i)}^*\big) \big)\vec{\omega}_i^{\vphantom{*}}$
			\Comment{Use matvecs}
			\EndFor
			\State $\hat{\tr} \gets (m/2)^{-1} \sum_{i=1}^{m/2} \hat{\tr}_i$
			\State $\hat{\err}^2 \gets ((m/2)(m/2 - 1))^{-1} \sum_{i=1}^{m/2} (\hat{\tr}_i - \hat{\tr})^2$
		\end{algorithmic}
	\end{algorithm}
	
	\subsubsection{The \NysTrace estimator}
	\label{sec:nystrace-intro}
	
	Our second method, \NysTrace, is an exchangeable trace estimator designed
	for positive-semidefinite (psd) matrices.  Rather than using a randomized SVD to
	reduce the variance, this estimator uses a Nystr{\"o}m
	approximation~\cite[\S14]{MT20a} of the psd matrix $\mat{A} \in \real^{N \times N}$.
	The Nystr\"om approximation takes the form
	\begin{equation} \label{eq:nys}
		\mat{A}\langle \mat{X}\rangle \coloneqq \mat{A}\mat{X} (\mat{X}^*\mat{A}\mat{X})^\dagger (\mat{A}\mat{X})^*
		\quad\text{for a test matrix $\mat{X} \in \real^{N\times k}$.}
	\end{equation}
	The Nystr{\"o}m method requires only $k$ matvecs to compute a rank-$k$ approximation, while
	the randomized SVD requires $2k$ matvecs.
	
	Let us summarize the \NysTrace method.  Draw iid isotropic test vectors $\vec{\omega}_1, \dots, \vec{\omega}_m$,
	and form the test matrix $\mat{\Omega} = \begin{bmatrix} \vec{\omega}_1 & \dots & \vec{\omega}_m \end{bmatrix}$.
	The basic estimators take the form
	\begin{equation} \label{eq:XNys-i}
		\hat{\tr}_i \coloneqq \tr \mat{A}\langle \mat{\Omega}_{-i} \rangle + \vec{\omega}_i^* \big((\mat{A} - \mat{A}\langle \mat{\Omega}_{-i}\rangle)\vec{\omega}_i^{\vphantom{*}} \big)
		\quad\text{for $i = 1, \dots, m$.}
	\end{equation}
	As usual, $\mat{\Omega}_{-i}$ denotes the test matrix with the $i$th column removed.
	To obtain the \NysTrace estimator and an error estimate, we use the formulas
	\begin{equation} \label{eq:xnystrace}
		\hat{\tr}_{\rm XN} \coloneqq \frac{1}{m} \sum_{i=1}^m \hat{\tr}_i
		\quad\text{and}\quad
		\hat{\err}_{\rm XN} \coloneqq \frac{1}{m(m-1)} \sum_{i=1}^m (\hat{\tr}_i - \hat{\tr}_{\rm XN})^2.
	\end{equation}
	The \NysTrace estimator is unbiased and exchangeable.  \Cref{thm:main} provides a bound
	for the variance.
	See \cref{alg:xnystrace_naive} for na{\"i}ve \NysTrace pseudocode
	and \cref{sec:nystrace} for a more efficient approach.
	
	The recent paper~\cite{PCK22} describes an estimator called \NysPP that uses a Nystr{\"o}m approximation to perform reduced-variance trace estimation.  \NysPP violates the exchangeability principle,
	while \NysTrace repairs this weakness.
	
	\begin{algorithm}[t]
		\caption{\NysTrace: Na{\"i}ve implementation \label{alg:xnystrace_naive}}
		\begin{algorithmic}[1]
			\Require Psd matrix $\mat{A} \in \real^{N \times N}$ and number $m$ of matvecs
			\Ensure Trace estimate $\hat{\tr} \approx \tr \mat{A}$ and error estimate $\hat{\err} \approx \vert \hat{\tr} - \tr \mat{A} \vert$
			\State Draw iid isotropic $\vec{\omega}_1,\ldots,\vec{\omega}_m \in \real^N$
			\Comment See \cref{sec:distribution}
			\State $\mat{\Omega} \leftarrow \begin{bmatrix} \vec{\omega}_1 & \dots & \vec{\omega}_m\end{bmatrix}$
			\State $\mat{Y} \leftarrow \mat{A}\mat{\Omega}$
			\Comment{Use matvecs}
			\For{$i = 1,2,3,\ldots,m$}
			\State $\mat{\hat{A}}_i \leftarrow \mat{Y}_{-i}^{\vphantom{*}}(\mat{\Omega}_{-i}^*\mat{Y}_{-i}^{\vphantom{*}})^\dagger \mat{Y}_{-i}^*$
			\Comment{Remove $i$th column of $\mat{Y}$ and $\mat{\Omega}$}
			\State $\hat{\tr}_i \leftarrow \tr \mat{\hat{A}}_i + \vec{\omega}_i^*((\mat{A}-\mat{\hat{A}}_i)\vec{\omega}_i^{\vphantom{*}})$
			\Comment{Use matvecs}
			\EndFor
			\State $\hat{\tr} \gets m^{-1} \sum_{i=1}^m \hat{\tr}_i$
			\State $\hat{\err}^2 \gets (m(m-1))^{-1} \sum_{i=1}^m (\hat{\tr}_i - \hat{\tr})^2$
		\end{algorithmic}
	\end{algorithm}

	\subsubsection{Stochastic diagonal estimators}
	
	As an extension of \textsc{XTrace}\xspace, we also propose the \XDiag algorithm for estimating the \emph{diagonal} of an implicitly defined matrix.
	We will discuss this approach in \cref{sec:xdiag}.
	
	\subsection{Numerical experiments} \label{sec:numerical_experiments}
	
	To highlight the advantages of the \textsc{XTrace}\xspace and \NysTrace estimators,
	we present some motivating numerical experiments.  \Cref{sec:experiments}
	contains further numerical work.
	
	\subsubsection{Exploiting spectral decay} 
	\label{sec:intro-synth-exper}
	
	Our first experiment uses a synthetic input matrix to illustrate how the
	exchangeable estimators wring more information out of the samples.
	
	We apply several trace estimators to a psd matrix with exponentially
	decreasing eigenvalues; see \cref{sec:comparisons} for the details of the matrix.
	\Cref{fig:exp_standout} reports the average error over 1000 trials.
	The Girard--Hutchinson estimator (\Hutch) converges at the Monte Carlo rate,
	whereas the newer estimators all converge much faster.  This
	improvement comes from variance reduction techniques that
	exploit the spectral decay.
	Observe that \textsc{XTrace}\xspace and \NysTrace converge exponentially fast at $1.5\times$ and $3\times$ the rate of \HutchPP, until reaching machine precision.
	For a fixed budget of $m$ matvecs, \textsc{XTrace}\xspace and \NysTrace can reduce the error by several orders of magnitude compared to \HutchPP.
	Strikingly, the reduction in variance from enforcing exchangeability is almost as significant as the reduction in variance from using a low-rank approximation as a control variate.
	
	\begin{figure}[t]
		\centering
		\includegraphics[width=0.6\textwidth]{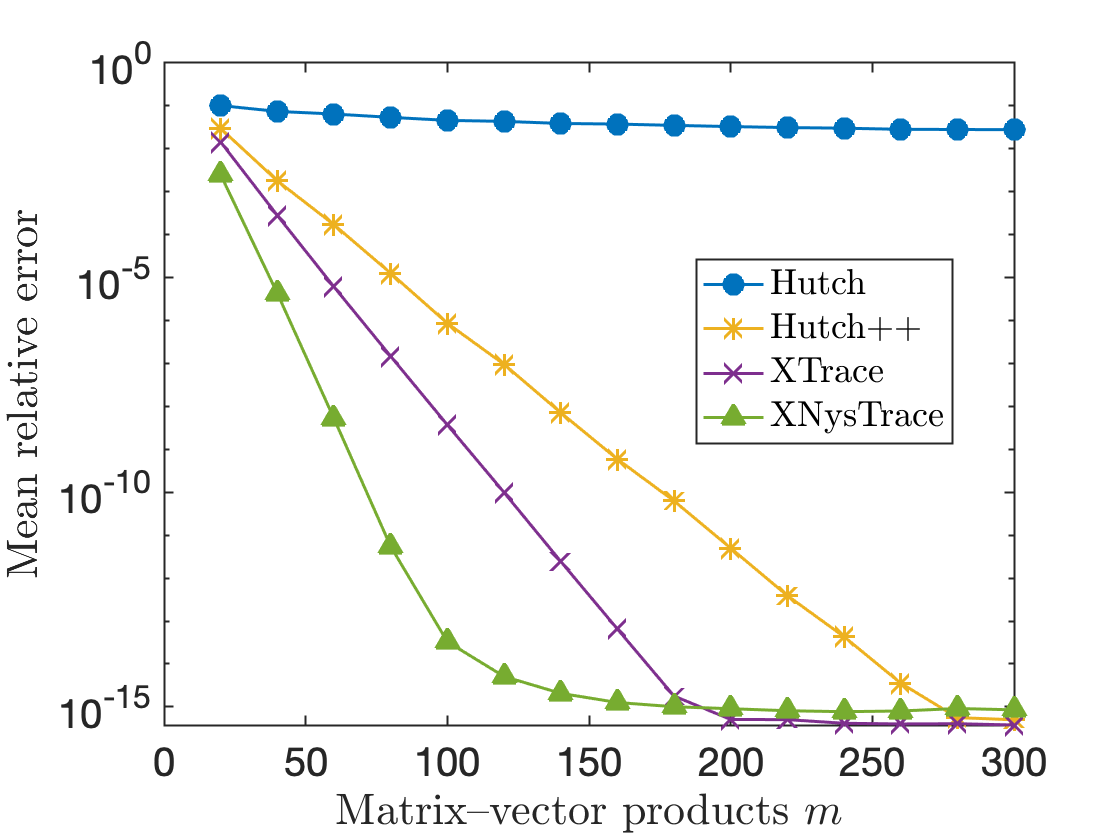} 
		
		\caption{\textbf{Exploiting spectral decay.}  Average error of trace estimators applied to a (synthetic) psd matrix with exponentially decreasing eigenvalues. See \cref{sec:intro-synth-exper}.} \label{fig:exp_standout}
	\end{figure}
	
	\subsubsection{Computing partition functions}
	\label{sec:intro-quantum-exper}
	
	Our second experiment shows how the advantages of using
	exchangeable estimators persist in a scientific application.
	
	We apply several trace estimators to compute the partition function for a quantum system
	\begin{equation*}
		Z(\beta) \coloneqq \tr \exp(-\beta \mat{H}),
	\end{equation*}
	where $\mat{H}$ is a symmetric Hamiltonian matrix and $\beta > 0$ is an inverse temperature.
	Specifically, we consider the Hamiltonian matrix $\mat{H}$ for the tranverse-field Ising model on $18$ sites, which has dimension $N = 2^{18} = 262\,144$.
	See \cref{sec:quantum} for details on the matrix $\mat{H}$ and the partition function $Z(\beta)$.
	To evaluate matvecs with $\exp(-\beta\mat{H})$, we use the code of Higham \cite{Hig10} that implements an adaptive polynomial approximation \cite{AH11}.
    For this problem, each matvec is expensive, occupying over 98\% of the total computation time for \textsc{XTrace}\xspace with $m=100$. 
    Thus, the cost of all the trace estimators is dominated by the computational cost of the matvecs.
	
	\begin{figure}[t]
		\centering
		\includegraphics[width=0.6\textwidth]{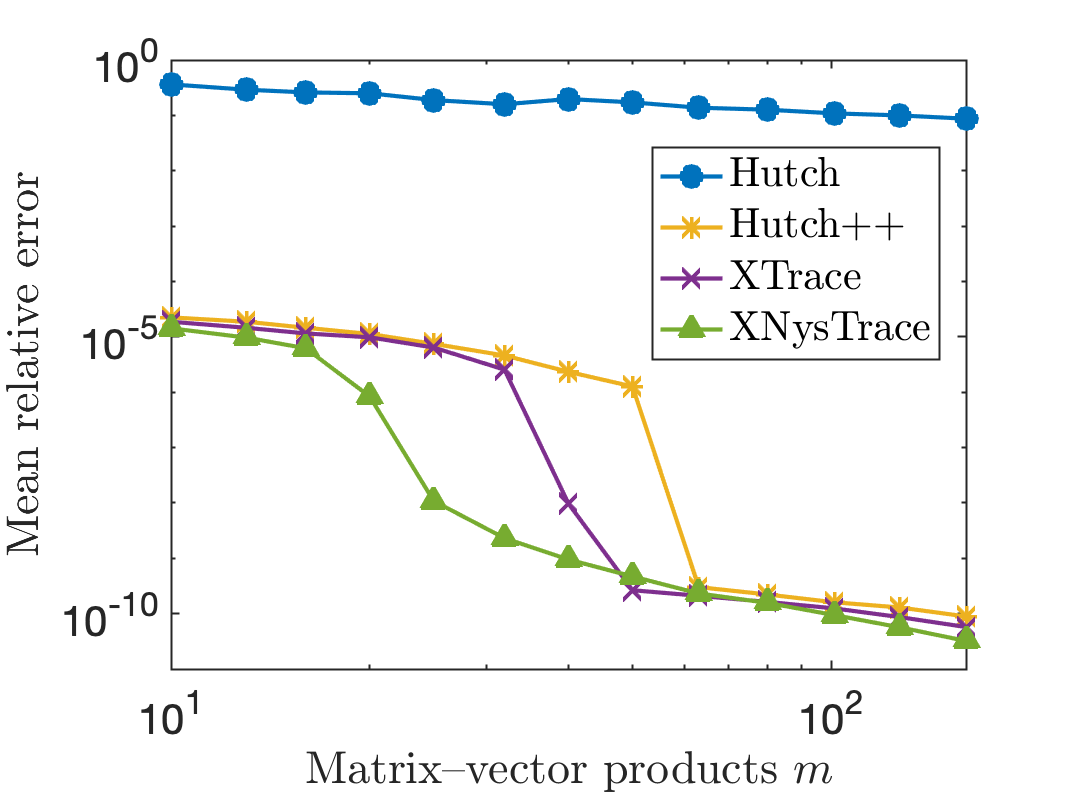} 
		
		\caption{\textbf{Computing a partition function.}  Average error of trace estimators applied to the partition function of the transverse-field Ising model. See \cref{sec:intro-quantum-exper}.} \label{fig:tfim_single_shot}
	\end{figure}
	
	\Cref{fig:tfim_single_shot} reports the mean estimation error over 100 trials.
	With just $m = 10$ matvecs, all the variance-reduced methods achieve errors that are $5$ orders of magnitude
	smaller than the Girard--Hutchinson estimator.
	Furthermore, \textsc{XTrace}\xspace and \NysTrace converge more quickly than \HutchPP as we increase the number of matvecs.
	For example, with $m = 40$ matvecs, \textbf{\textsc{XTrace}\xspace is 240$\times$ more accurate}, and \textbf{\NysTrace is 2400$\times$ more accurate}.
    The disparate performance of the three variance-reduced methods is a consequence of the eigenvalues of the matrix $\exp(-\beta\mat{H})$, which drop sharply from eigenvalue $19$ to $20$: $\lambda_{19}/\lambda_{20} = 4600$.
    Because \NysTrace, \textsc{XTrace}\xspace, and \HutchPP use approximation ranks of roughly $m$, $m/2$, and $m/3$, it takes these algorithms 20, 40, and 60 matvecs to capitalize on this eigenvalue drop, leading to the significant differences in performance between the methods. 
	
	\subsection{Theoretical guarantees}
	
	To explain the excellent performance of the exchangeable estimators,
	we establish detailed theoretical guarantees.
	For theoretical convenience, our analysis uses standard
	normal test vectors.  As a consequence, we can deliver
	explicit constants that allow us to make meaningful
	comparisons among the methods.  

	\begin{theorem}[Variance bounds] \label{thm:main}
		Let $\mat{A} \in \real^{N \times N}$ be a square matrix.
		Fix the number of matvecs: $m \geq 8$.
		The \HutchPP, \textsc{XTrace}\xspace, and \NysTrace estimators are unbiased estimators
		of the trace.  With standard normal test vectors $\vec{\omega} \sim \textsc{normal}(\vec{0}, \Id)$,
		these estimators satisfy the variance bounds
		\begin{align*}
			\bigl(\mathbb{E} \bigl|\hat{\tr}_{\rm H++} - \tr \mat{A} \bigr|^2\bigr)^{\tfrac{1}{2}}
			&\leq \makebox[\widthof{$\sqrt{m}$}][r]{}\min_{\color{blue} r \leq m/3 - 2} \, 
			\Biggl( \sqrt{2} \frac{\lVert \mat{A} - \lowrank{\mat{A}}_r \rVert_{\rm F}}{\sqrt{m \slash 3 - r - 1}} \Biggr); \\
			\bigl(\mathbb{E} \bigl|\hat{\tr}_{\rm X} - \tr \mat{A} \bigr|^2\bigr)^{\tfrac{1}{2}}
			&\leq \makebox[\widthof{$\sqrt{m}$}][r]{$\sqrt{m}$}\min_{\color{blue} r \leq m \slash 2 - 4} \,
			\Biggl(2 \frac{\lVert \mat{A} - \lowrank{\mat{A}}_r \rVert}{\sqrt{m/2 - r - 3}} + 2 \e \frac{\lVert \mat{A} - \lowrank{\mat{A}}_r \rVert_{\rm F}}{m \slash 2 - r - 3} \Biggr); \\
			\bigl(\mathbb{E} \bigl|\hat{\tr}_{\rm XN} - \tr \mat{A} \bigr|^2\bigr)^{\tfrac{1}{2}}
			&\leq \makebox[\widthof{$\sqrt{m}$}][r]{$m$}\min_{\color{blue} r \leq m-6} \, 
			\Biggl(\sqrt{8} \frac{\lVert \mat{A} - \lowrank{\mat{A}}_r \rVert}{m - r - 5}
			+ \sqrt{2} \frac{\lVert \mat{A} - \lowrank{\mat{A}}_r \rVert_{\rm F}}{(m - r - 5)^{3/2}}\\
			&\qquad\qquad\qquad\qquad\qquad + 
			5 \e^2 \frac{\lVert \mat{A} - \lowrank{\mat{A}}_r\rVert_\ast}{(m-r-5)^2} \Biggr).
		\end{align*}
		These formulas involve the Frobenius norm $\norm{\cdot}_{\rm F}$, the spectral norm $\norm{\cdot}$,
		and the trace/nuclear norm $\norm{\cdot}_*$.  The matrix $\lowrank{\mat{A}}_r$ is a simultaneous best rank-$r$ approximation of $\mat{A}$ in these norms.
		
		In addition, for each of these three methods,
		it suffices to use $m = \order(\eta^{-1/2})$ matvecs
		to achieve the variance bound
		\begin{equation}
			\Var[\hat{\tr}] = \expect \big[ | \hat{\tr} - \tr \mat{A} |^2 \big]
			\leq \eta \lVert \mat{A} \rVert_\ast^2
			\quad\text{for $\eta \in (0,1)$.}
		\end{equation}
		
	\end{theorem}
	The proof of \cref{thm:main} appears in \cref{sec:proofs}.
	
	As the number $m$ of matvecs increases, \cref{thm:main} ensures that the variance of \textsc{XTrace}\xspace, \NysTrace, and \HutchPP decreases at a rate of $\order(1/m^2)$.  Therefore, these algorithms are all superior to the Girard--Hutchinson estimator, whose variance decreases at the Monte Carlo rate $\Theta(1/m)$.
	
	\Cref{thm:main} also demonstrates the advantage of \textsc{XTrace}\xspace and \NysTrace for matrices whose singular values decay rapidly.  This benefit is visible from the error bounds because they allow for larger values $r$ of the approximation rank.  As an example, consider a psd matrix $\mat{A}$ whose eigenvalues have exponential decay with rate $\alpha \in (0,1):$
	\begin{equation*}
		\lambda_i(\mat{A}) \le \alpha^i \quad \text{for } i =1,2, 3, \ldots.
	\end{equation*}
	The errors of \HutchPP, \textsc{XTrace}\xspace, and \NysTrace decay like
	\begin{align*}
		\bigl(\mathbb{E} \bigl|\hat{\tr}_{\rm H++} - \tr \mat{A} \bigr|^2\bigr)^{1 \slash 2} &\le \makebox[\widthof{$\sqrt{m}$}][r]{}\,C_1(\alpha)\, \alpha^{\color{blue} m/3}; \\ \bigl(\mathbb{E} \bigl|\hat{\tr}_{\rm X} - \tr \mat{A} \bigr|^2\bigr)^{1 \slash 2}&\le \sqrt{m} \,C_2(\alpha)\,\alpha^{\color{blue} m/2}; \\
		\bigl(\mathbb{E} \bigl|\hat{\tr}_{\rm XN} - \tr \mat{A} \bigr|^2\bigr)^{1 \slash 2}&\le \makebox[\widthof{$\sqrt{m}$}][r]{$m$} \,C_3(\alpha)\, \alpha^{\color{blue} m}.
	\end{align*}
	For this class of matrix, \textsc{XTrace}\xspace converges exponentially fast at $1.5\times$ the rate of \HutchPP, and \NysTrace converges exponentially fast at $3 \times$ the rate of \HutchPP.
	This is precisely the behavior we observe in \cref{fig:exp_standout}.
	
	\subsection{Benefits}
	
	In summary, the \textsc{XTrace}\xspace and \NysTrace estimators have several desirable features
	as compared with previous approaches.
	
	\begin{enumerate}
		\item \textbf{Higher accuracy:} For a budget of $m$ matvecs, \textsc{XTrace}\xspace and \NysTrace often yield errors that are \emph{orders of magnitude smaller} than \HutchPP. 
		\item \textbf{Efficient algorithms:}  We have designed implementations of \textsc{XTrace}\xspace and \NysTrace that only require $m$ matvecs plus $\order(m^2 N)$ arithmetic operations, which is the same computational cost as \HutchPP. 
		\item \textbf{Error estimation:} We can equip the \textsc{XTrace}\xspace and \NysTrace estimators with reliable error estimates.
	\end{enumerate}
	
	\subsection{A brief history of stochastic trace estimation}
	\label{sec:related_work}
	
	Girard wrote the first paper~\cite{Gir89} on stochastic trace estimation, in which he proposed
	the estimator~\cref{eq:girard_hutchinson} with test vectors drawn uniformly from a
	Euclidean sphere.
	His goal was to develop an efficient way to perform generalized
	cross-validation for smoothing splines.  Hutchinson~\cite{Hut89} suggested using
	random sign vectors instead: $\vec{\omega} \sim \textsc{uniform}\{\pm1\}^N$.
	See~\cite[\S4]{MT20a} for further details.
	
	In the last five years, researchers have developed far more efficient methods for trace
	estimation by incorporating variance reduction techniques.
	In 2017, Saibaba, Alexanderian, and Ipsen \cite{SAI17} proposed a biased estimator that outputs the trace of a low-rank approximation as a surrogate for $\tr \mat{A}$.  Around the same time, Gambhir, Stathopoulos, and Orginos \cite{GSO17} proposed a hybrid estimator, similar to \HutchPP, that outputs the trace of a low-rank approximation, $\tr \mat{\hat{A}}$, plus a Girard--Hutchinson estimate for $\tr (\mat{A}-\mat{\hat{A}})$.  The paper \cite{Lin17} of Lin contains related ideas.
	
	In 2021, Meyer et al.~\cite{MMMW21}  distilled the ideas from \cite{GSO17,Lin17} to develop the \HutchPP algorithm.  They proved that \HutchPP satisfies a worst-case variance bound of $\order(1/m^2)$.
	Meyer et al.~also proposed a version of \HutchPP that needs only a single pass
	over the input matrix.  The follow-up paper \cite{JPWZ21a} sharpens the analysis of the single-pass algorithm.  
	
	Persson, Cortinovis, and Kressner \cite{PCK22} have introduced several refinements to the \HutchPP estimator. 
	Their first improvement adaptively apportions test vectors between approximating the matrix
	and estimating the trace of the residual in order to meet an error tolerance. 
	Their second contribution is \NysPP, a version of \HutchPP for psd matrices that uses
	Nystr\"om approximation.
	
	\textsc{XTrace}\xspace and \NysTrace build on the previous strategies of variance reduction using low-rank approximation.
	However,
	\textsc{XTrace}\xspace and \NysTrace take a step forward by also enforcing the exchangeability principle.
	These algorithms push the ideas of \HutchPP and \NysPP to their limit by using \emph{all} the test vectors for low-rank approximation and \emph{all} the test vectors for residual trace estimation.
	
	To conclude, let us mention two techniques designed for computing the trace of a standard matrix function (that is, $\tr f(\mat{A})$).
	First, stochastic Lanczos quadrature \cite{UCS17} approximates the spectral density of $\mat{A}$,
	from which estimates of $\tr f(\mat{A})$ for any function $f(\cdot)$ are immediately accessible.
	See \cite{CTU22} for a recent overview of stochastic Lanczos quadrature and related ideas.
	As a second approach, when the matvecs $\vec{\omega} \mapsto f(\mat{A})\vec{\omega}$
	are computed using a Krylov subspace method \cite[{\S}13.2]{Hig08}, the paper
	\cite{CH22} recommends reuse of the matvecs from the Krylov subspace method
	for the purpose of trace estimation.
	
	\subsection{Reproducible research}
	Optimized MATLAB R2022b implementations of our algorithms as well as code to reproduce the experiments in this paper can be found online at \url{https://github.com/eepperly/XTrace}.
	
	\subsection{Outline}
	The remainder of the paper is organized as follows.
	\Cref{sec:estimators} describes efficient implementations for \textsc{XTrace}\xspace, \NysTrace, and the diagonal estimator \XDiag.
	\Cref{sec:extensions} discusses error estimation and adaptive stopping,
	\cref{sec:experiments} presents numerical experiments, and \cref{sec:proofs} proves our theoretical results.
	
	\subsection{Notation}
	Matrices and vectors are denoted by capital and lowercase bold letters.
	The $i$th column of $\mat{B}$ is expressed as $\vec{b}_i$, and the $(i,j)$th entry of $\mat{B}$ is $b_{ij}$.
	For a matrix $\mat{B}$, we form a matrix $\mat{B}_{-i}$ by deleting the $i$th column from $\mat{B}$.
	Similarly, we form $\mat{B}_{-ij}$ by deleting the $i$th and $j$th columns.
	We work with the spectral norm $\norm{\cdot}$, the Frobenius norm $\norm{\cdot}_{\rm F}$,
	and the trace/nuclear norm $\norm{\cdot}_*$.
	The symbol $\lowrank{\mat{B}}_r$ denotes a (simultaneous) best rank-$r$
	approximation of $\mat{B}$ with respect to any unitarily invariant norm.
	
	\section{Efficient implementation of exchangeable trace estimators} \label{sec:estimators}
	
	This section works through some issues that arise in the implementation
	of the \textsc{XTrace}\xspace and \NysTrace estimators.
	\Cref{sec:xtrace,sec:nystrace} show how to implement the
	new estimators efficiently using insights from numerical
	linear algebra.
	\Cref{sec:distribution}
	discusses a method of renormalizing the test vectors that improves the accuracy of \textsc{XTrace}\xspace and \NysTrace.
	\Cref{sec:xdiag} develops the \XDiag
	estimator.
	
	\subsection{Computing \textsc{XTrace}\xspace} \label{sec:xtrace}
	In this section, we develop an efficient implementation of the \textsc{XTrace}\xspace estimator
	from \cref{sec:xtrace-intro}.
	Recall that $\mat{A} \in \real^{N \times N}$ is a general square matrix,
	and introduce the test matrix
	$\mat{\Omega} = \begin{bmatrix}\vec{\omega}_1 & \ldots & \vec{\omega}_{m/2}\end{bmatrix} \in \real^{N \times (m/2)}$.
	
	First, we form the matrix product $\mat{A}\mat{\Omega}$ and compute the orthogonal decomposition
	$\mat{A}\mat{\Omega} = \mat{Q}\mat{R}$.
	Following \cite[App.~A.2]{ET22}, 
	we make the critical observation that the basis matrix
	$\mat{Q}_{(i)} = \orth(\mat{A} \mat{\Omega}_{-i})$
	is related to the full basis matrix $\mat{Q}$ by a rank-one update:
	\begin{equation} \label{eq:Q_update}
		\mat{Q}_{(i)}^{\vphantom{*}} \mat{Q}^*_{(i)} = \mat{Q} \big( \Id - \vec{s}_i^{\vphantom{*}}\vec{s}_i^*\big) \mat{Q}^*
		\quad\text{where}\quad
		\mat{R}_{-i}^* \vec{s}_i = \vec{0}
		\quad\text{and}\quad
		\norm{\vec{s}_i}_{\ell_2} = 1.
	\end{equation}
	Thus, the rank-one update requires a unit vector $\vec{s}_i \in \real^{N}$ in the null space of $\mat{R}_{-i}^{*}$.
	
	Let us exhibit an efficient algorithm that simultaneously computes all the vectors $\vec{s}_i$ for $1 \leq i \leq m/2$.
	We argue that the matrix $\mat{S} = \begin{bmatrix} \vec{s}_1 & \dots & \vec{s}_{m/2} \end{bmatrix}$
	can be represented as
	\begin{equation} \label{eq:S_XTrace}
		\mat{S} = (\mat{R}^{*})^{-1} \mat{D},
	\end{equation}
	where the diagonal matrix $\mat{D}$ chosen to enforce the normalization of the columns of $\mat{S}$.  Indeed, since $\mat{R}^*\mat{S} = \mat{D}$ is diagonal, the $i$th column of $\mat{R}_{-i}^*\mat{S}$ is the zero vector.  We reach the desired conclusion $\mat{R}_{-i}^*\vec{s}_i = \vec{0}$.
	
	In summary, given the full basis $\mat{Q}$, we can use \cref{eq:S_XTrace} to compute all the vectors $\vec{s}_i$ needed to construct the orthogonal projectors $\mat{Q}_{(i)}^{\vphantom{*}} \mat{Q}_{(i)}^\ast$ for $1 \leq i \leq m/2$ appearing in~\cref{eq:Q_update}.  This calculation requires just $\order(m^3)$ operations, which is dominated by the cost of solving $m/2$ triangular linear systems.
	It follows that the \textsc{XTrace}\xspace estimator can be computed in just $\order(m^2N)$ operations, which is the same asymptotic cost as \HutchPP. For full details, see the efficient MATLAB implementation in the supplementary materials \cref{list:xtrace}.
	
	\subsection{Computing \NysTrace} \label{sec:nystrace}
	
	We can design a similar method to compute the \NysTrace estimator \cref{eq:xnystrace} efficiently.
	Let $\mat{A} \in \real^{N\times N}$ be a psd matrix and define the test matrix $\mat{\Omega} = \begin{bmatrix} \vec{\omega}_1 & \dots & \vec{\omega}_m \end{bmatrix}$.
	
	As before, we compute the orthogonal decomposition $\mat{A}\mat{\Omega} = \mat{Q}\mat{R}$.
	We can express the Nystr\"om approximation \cref{eq:nys} in the form
	\begin{equation*}
		\nys{\mat{A}}{\mat{\Omega}} = \mat{Q}\mat{R} \mat{H}^{-1} \mat{R}^*\mat{Q}^*,
	\end{equation*}
	where $\mat{H} = \mat{\Omega}^*\mat{A}\mat{\Omega}$.
	After deleting the $i$th column from $\mat{\Omega}$,
	the resulting Nystr\"om approximation satisfies
	\begin{equation} \label{eq:nys_one_deleted}
		\nys{\mat{A}}{\mat{\Omega}_{-i}} = \mat{Q} \mat{R}_{-i}^{\vphantom{*}}\mat{H}_{(i)}^{-1}\mat{R}_{-i}^* \mat{Q}^*,
	\end{equation}
	where $\mat{H}_{(i)}$ is $\mat{H}$ upon deletion of its $i$th row and column.
	To compute $\nys{\mat{A}}{\mat{\Omega}_{-i}}$ efficiently, we recognize that \cref{eq:nys_one_deleted} can be expressed as a rank-one update:
	\begin{equation}
		\label{eq:rank_one}
		\nys{\mat{A}}{\mat{\Omega}_{-i}} = \mat{Q} \mat{R} \mleft( \mat{H}^{-1} - \frac{\mat{H}^{-1} \mathbf{e}_i^{\vphantom{*}}\evec_i^*\mat{H}^{-1}}{\evec_i^*\mat{H}^{-1}\mathbf{e}_i^{\vphantom{*}}} \mright) \mat{R}^* \mat{Q}^*.
	\end{equation}
	Taking advantage of the rank-one update formula \cref{eq:rank_one}, we can form the \NysTrace estimator using just $m$ matvecs and $\order(m^2N)$ post-processing operations.
	An efficient MATLAB implementation appears in the supplementary materials \cref{list:xnystrace}, which incorporates several additional tricks taken from \cite{LLS+17,TYUC17b} to improve its numerical stability.
	
	\subsection{Normalization of test vectors}
	\label{sec:distribution}
	
	For the best general performance of \textsc{XTrace}\xspace and \NysTrace, we recommend a modification of the basic \textsc{XTrace}\xspace and \NysTrace procedure in which the test vectors are orthogonalized against the low-rank approximation $\mat{Q}_{(i)}^{\vphantom{*}}\mat{Q}_{(i)}^*\mat{A}$ or $\mat{A}\langle \mat{\Omega}_{-i}\rangle$ and renormalized.
    This renormalization strategy proceeds as follows:
	First, we draw the test vectors from a spherically symmetric distribution,
	such as $\vec{\omega} \sim \textsc{normal}( \vec{0}, \Id )$.
	We use these test vectors to form the matrix $\mat{Q}_{(i)}$ or Nystr\"om approximation $\mat{A} \langle \mat{\Omega}_{-i}\rangle$.
	Second, when computing the basic trace estimates, we \textbf{normalize} the test vectors.  In \textsc{XTrace}\xspace, we compute
	$$
	\vec{\mu}_i \coloneqq \bigl(\Id - \mat{Q}_{(i)} \mat{Q}_{(i)}^*\bigr) \vec{\omega}_i
	\quad\text{and}\quad
	\vec{\nu}_i \coloneqq \sqrt{N - \rank(\mat{Q}_{(i)})} \cdot \frac{\vec{\mu}_i}{\norm{ \vec{\mu}_i }_{\ell_2}}.
	$$
	To obtain the $i$th trace estimate, we form
	$$
	\hat{\tr}_{i} \coloneqq \tr\big(\mat{Q}_{(i)}^* (\mat{A}\mat{Q}_{(i)}^{\vphantom{*}})\big)
	+ \vec{\nu}_i^*\big(\mat{A} \vec{\nu}_i^{\vphantom{*}}\big). 
	$$
	For \NysTrace, we set 
	\begin{equation*}
		\mat{P}_{(i)} \coloneqq \orth \mat{\Omega}_{-i}, \:\: \vec{\mu}_i \coloneqq (\Id - \mat{P}_{(i)}\mat{P}_{(i)}^*) \vec{\omega}_i, \:\:\text{and}\:\: \vec{\nu}_i \coloneqq\sqrt{N - \rank(\mat{P}_{(i)})} \cdot \frac{\vec{\mu}_i}{\norm{ \vec{\mu}_i }_{\ell_2}}.
	\end{equation*}
	Then define the basic trace estimates
	\begin{equation*}
		\hat{\tr}_i \coloneqq \tr \mat{A}\langle \mat{\Omega}_{-i}\rangle + \vec{\nu}_i^* \big( \mat{A}\vec{\nu}_i^{\vphantom{*}} \big).
	\end{equation*}
	The normalization removes a source of variance related to the random lengths of the vectors $\vec{\mu}_i$, improving the accuracy compared to unnormalized Gaussian test vectors or uniform random vectors on the sphere.
	We compare this normalization approach against alternative distributions for test vectors in \cref{sec:test_vectors}.
	
	\subsection{Diagonal estimation} \label{sec:xdiag}
	
	In the spirit of Girard and Hutchinson, the paper~\cite{BKS07} of Bekas, Kokiopoulou, and Saad (BKS) develops an estimator for the diagonal of an implicit matrix:
	\begin{equation} \label{eq:hutch_diag}
		\hat{\vec{\diag}}_{\rm BKS} = \frac{\sum_{i=1}^m \vec{\omega}_i \odot (\mat{A}\vec{\omega}_i)}{\sum_{i=1}^m \vec{\omega}_i \odot \vec{\omega}_i}.
	\end{equation}
	Here, $\odot$ denotes the entrywise product and the division is performed entrywise.
	The recent paper~\cite{BN22} proposes a biased estimator for the diagonal, called \DiagPP, that is inspired by BKS and \HutchPP.
	
	We have observed that the exchangeability principle
	leads to an unbiased diagonal estimator with lower variance.
	Our diagonal estimator, \XDiag, takes the form
	\begin{equation*}
		\vec{\hat{\diag}}_{\rm X} = \frac{1}{m/2} \sum_{i=1}^{m/2} \mleft[ \diag\big(\mat{Q}_{(i)}^{\vphantom{*}}(\mat{Q}_{(i)}^*\mat{A})\big) + \frac{\vec{\omega}_i \odot \big(\Id - \mat{Q}_{(i)}^{\vphantom{*}}\mat{Q}_{(i)}^*\big)(\mat{A}\vec{\omega}_i)}{\vec{\omega}_i \odot \vec{\omega}_i} \mright],
	\end{equation*}
	where $\mat{Q}_{(i)}$ is defined in \cref{eq:Q_i}.
	In contrast to \textsc{XTrace}\xspace, the \XDiag estimator requires matvecs with $\mat{A}^*$ in addition to matvecs with $\mat{A}$.
	The same ideas from \cref{sec:xtrace} allow us to implement \XDiag in $\order(m^2N)$ operations; an implementation is provided in the supplementary materials (\cref{list:xdiag}).
	
	\section{Error estimation and adaptive stopping} \label{sec:extensions}
	
	Our exchangeable estimators depend on averaging over a family of basic estimators, and we can reuse the basic estimators to compute a reliable posterior approximation for the error (\cref{sec:error}).
	The error estimate allows us to develop adaptive
	methods for selecting the number $m$ of matvecs to achieve a specified error tolerance
	(\cref{sec:adaptive_stopping}).  
	These refinements are very important for practical
	implementations.
	
	\subsection{Error estimation} \label{sec:error}
	
	The \textsc{XTrace}\xspace and \NysTrace estimators are both formed as averages of individual trace estimates $\hat{\tr}_1,\ldots,\hat{\tr}_\ell$ where
	\begin{align*}
		\hat{\tr}_i &= \tr\big(\mat{Q}_{(i)}^*\mat{A}\mat{Q}_{(i)}^{\vphantom{*}}\big) + \vec{\omega}_i^*\big(\Id-\mat{Q}_{(i)}^{\vphantom{*}}\mat{Q}_{(i)}^*\big) \mat{A}\big(\Id-\mat{Q}_{(i)}^{\vphantom{*}}\mat{Q}_{(i)}^*\big)\vec{\omega}_i^{\vphantom{*}},& &\!\!\ell \coloneqq \frac{m}{2}& &\!\!\!\!\text{(\textsc{XTrace}\xspace);} \\
		\hat{\tr}_i &=\tr (\mat{A}\langle \mat{\Omega}_{-i}\rangle) + \vec{\omega}_i^* (\mat{A} - \mat{A}\langle \mat{\Omega}_{-i}\rangle)\vec{\omega}_i^{\vphantom{*}},& &\!\!\ell \coloneqq m & &\!\!\!\!\text{(\NysTrace).}
	\end{align*}
	The scaled variance of the individual trace estimates $\hat{\tr}_i$
	provides a useful estimate for the squared error in the trace estimate:
	\begin{equation}
		\label{eq:error_est}
		\hat{\err}^2 \coloneqq \frac{1}{\ell(\ell-1)} \sum_{i=1}^\ell \mleft| \hat{\tr}_i - \hat{\tr} \mright|^2 \approx \mleft| \tr(\mat{A}) - \hat{\tr} \mright|^2
		\quad\text{where}\quad
		\hat{\tr} = \frac{1}{\ell} \sum_{i=1}^\ell \hat{\tr}_i.
	\end{equation}
	The next result contains an analysis of this posterior error estimator.
	
	\begin{proposition}[Error estimate] \label{prop:estim_bound}
		The error estimate \cref{eq:error_est} satisfies
		\begin{equation*}
			\expect \hat{\err}^2 = \frac{1-\Cor(\hat{\tr}_1,\hat{\tr}_2)}{1+(\ell-1)\Cor(\hat{\tr}_1,\hat{\tr}_2)} \cdot \expect \mleft|  \tr(\mat{A}) - \hat{\tr} \mright|^2.
		\end{equation*}
		We have written $\Cor(\cdot, \cdot)$ for the correlation of two random variables.
	\end{proposition}
	
	The proof and some additional discussion of the correlation $\Cor(\hat{\tr}_1,\hat{\tr}_2)$  appears in \cref{sec:estim_bound}.  In practice, we find that the individual trace estimators have a small positive correlation, so we typically \emph{underestimate} the true error by a small amount.  Thus, the posterior error estimate
	is a valuable tool.  For an illustration, see \cref{fig:tfim_errs} in \cref{sec:quantum}.
	
	\subsection{Adaptive stopping} \label{sec:adaptive_stopping}
	
	In practice, we often wish to choose the number $m$ of matvecs adaptively
	to estimate $\tr \mat{A}$ up to a prescribed accuracy level:
	\begin{equation*}
		|\,{\tr(\mat{A}) - \hat{\tr}}\,| \leq \varepsilon \cdot |\,{\tr\mat{A}}\,|
		\quad\text{for $\epsilon \in (0,1)$.}
	\end{equation*}
	One simple way to achieve this tolerance is through a \emph{doubling strategy}:
	\begin{enumerate}
		\item To initialize, collect $m_0$ matvecs $\mat{A}\vec{\omega}_1,\ldots,\mat{A}\vec{\omega}_{m_0}$, and  set $j \leftarrow 0$.
		\item Use $\mat{A}\vec{\omega}_1,\ldots,\mat{A}\vec{\omega}_{m_j}$
		to form a trace estimate $\hat{\tr}^{(j)}$ and an error estimate $\hat{\err}^{(j)}$.
		\item If $\hat{\err}^{(j)} \leq \varepsilon \cdot |\hat{\tr}^{(j)}|$, then stop.
		\item Collect $m_j$ additional matvecs $\mat{A}\vec{\omega}_{m_j+1},\ldots,\mat{A}\vec{\omega}_{2m_j}$.
		Set $j \leftarrow j+1$ and $m_j \leftarrow 2m_{j-1}$. Go to step 2.
	\end{enumerate}
	The doubling strategy requires at most twice the optimal number of matvecs to meet the tolerance and maintains the $\order(m^2N)$ computational cost of \textsc{XTrace}\xspace and \NysTrace.
	We implement this approach in our experiments to produce \cref{fig:tfim_energies}.
	
	\section{Numerical experiments} \label{sec:experiments}
	
	This section presents a numerical evaluation of \textsc{XTrace}\xspace, \NysTrace, and \XDiag.
	\Cref{sec:comparisons} compares different trace estimators on synthetic matrices, \cref{sec:test_vectors} evaluates \NysTrace with different distributions for the test vectors,
	\cref{sec:quantum} applies \textsc{XTrace}\xspace and \NysTrace to computations in quantum statistical physics, and \cref{sec:networks} applies \XDiag to computations in network science.
    Throughout this section, we compare different trace estimators based on the error they achieve for a given budget of matvecs; experiments comparing the runtime of different trace estimators is provided in \cref{sec:time}.
	Code to reproduce the experiments can be found at \url{https://github.com/eepperly/XTrace}.
	
	\subsection{Comparison of trace estimators}
	\label{sec:comparisons}
	
	The first experiment is designed to compare the accuracy of six trace estimators that each use $m$ matvecs:
	\begin{itemize}
		\item \Hutch: The Girard--Hutchinson estimator \cref{eq:girard_hutchinson}.
		\item \LRA: The Saibaba et al.~\cite{SAI17} estimator (without additional subspace iteration): $\tr\mat{\hat{A}}$ where $\mat{\hat{A}} = \mat{Q}\mat{Q}^*\mat{A}$ and $\mat{Q} = \orth(\mat{A}\mat{\Omega})$.
		\item The \HutchPP estimator \cref{eq:hutch++}.
		\item The \NysPP estimator \cite{PCK22}.
		We use the implementation provided by the authors of \cite{PCK22}, modified to the test vector $\vec{\omega} \sim \textsc{uniform} \{\pm 1\}^N$.
		\item The \textsc{XTrace}\xspace estimator \cref{eq:xtrace}.
		\item The \NysTrace estimator \cref{eq:xnystrace}.
	\end{itemize}
	To create a fair comparison, we apply all six estimators using a test matrix whose entries are uniformly random signs: $\mat{\Omega} \sim \textsc{uniform}\{\pm 1\}^{N \times \ell}$, as was used in \HutchPP.
	The additional benefit for \textsc{XTrace}\xspace and \NysTrace of using normalized, spherically symmetric test vectors is explored in \cref{sec:test_vectors}.
	The supplementary materials (\cref{sec:compare_adapt}) contain additional comparisons with the adaptive \HutchPP algorithm of~\cite{PCK22}; this comparison requires a more complicated experimental setup.
	
	We apply each of these estimators to randomly generated matrices of the form
	\begin{equation*}
		\mat{A}(\vec{\lambda}) = \mat{U}\diag(\vec{\lambda})\mat{U}^*
	\end{equation*}
	where $\mat{U}$ is a Haar random orthogonal matrix. We use four choices for the eigenvalues $\vec{\lambda}$:
	\begin{itemize}
		\item \texttt{flat}: $\vec{\lambda} = (3 - 2(i-1)/(N-1) : i = 1,2,\ldots,N)$.
		\item \texttt{poly}: $\vec{\lambda} = (i^{-2} : i = 1,2,\ldots,N)$.
		\item \texttt{exp}: $\vec{\lambda} = (0.7^i : i = 0,1,\ldots,N-1)$.
		\item \texttt{step}: $\vec{\lambda} = (\underbrace{1,\ldots,1}_{\text{$50$ times}},\underbrace{10^{-3},\ldots,10^{-3}}_{\text{$N-50$ times}})$.
	\end{itemize}
	We fix the matrix dimension $N = 1000$, and we report the relative error $|{\tr(\mat{A})} - {\hat{\tr}}| / \tr(\mat{A})$ averaged over $1000$ trials.
	
	\begin{figure}[t]
		\centering
		\begin{subfigure}[b]{0.48\textwidth}
			\centering
			\includegraphics[width=\textwidth]{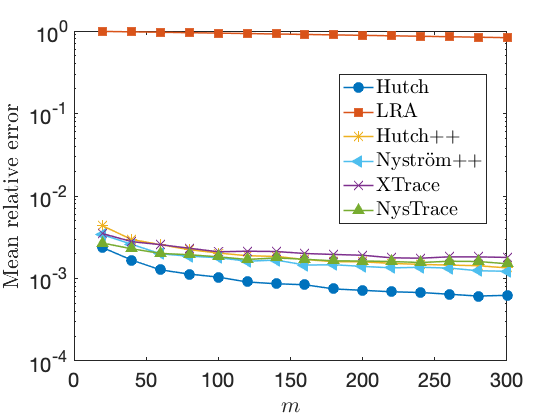}
			\caption{\texttt{flat}} \label{fig:flat}
		\end{subfigure}
		~
		\begin{subfigure}[b]{0.48\textwidth}
			\centering
			\includegraphics[width=\textwidth]{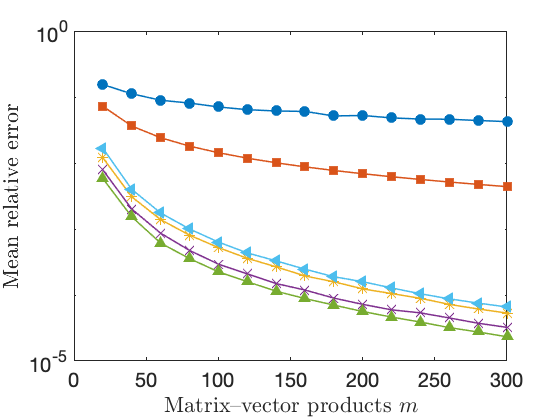}
			\caption{\texttt{poly}} \label{fig:poly}
		\end{subfigure}
		
		\begin{subfigure}[b]{0.48\textwidth}
			\centering
			\includegraphics[width=\textwidth]{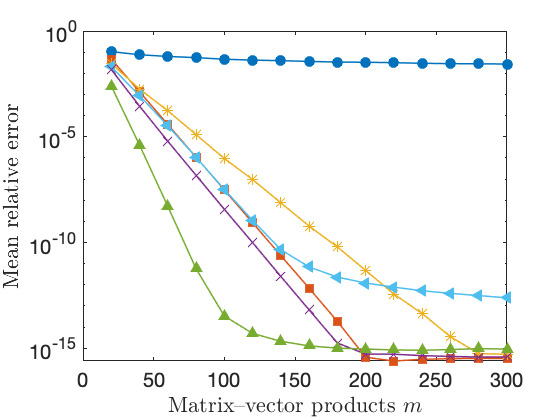}
			\caption{\texttt{exp}} \label{fig:exp}
		\end{subfigure}
		\begin{subfigure}[b]{0.48\textwidth}
			\centering
			\includegraphics[width=\textwidth]{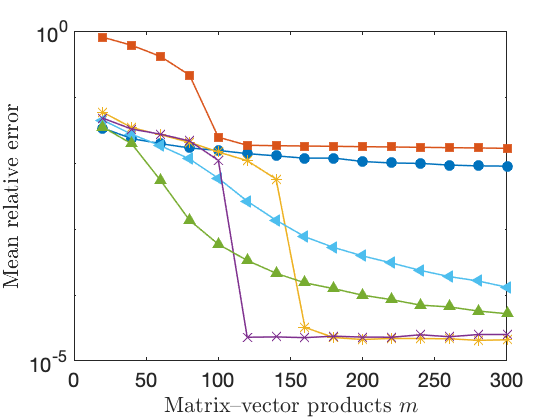}
			\caption{\texttt{step}} \label{fig:step}
		\end{subfigure}
		\caption{\textbf{Synthetic instances.}  Average relative error of trace estimators for matrices with particular spectral profiles using random sign test vectors. See~\cref{sec:comparisons} and \cref{fig:test_vectors}.} \label{fig:comparison}
	\end{figure}
	
	\textbf{Discussion.}
	The variance-reduced trace estimators dramatically outperform \Hutch, except on the \texttt{flat} instance (\cref{fig:flat}). 
	The implication is that \Hutch only makes sense when estimating the trace of a matrix with a nearly flat spectrum.
	For the \texttt{flat} instance, the performance of \LRA is especially poor because \LRA is a biased estimator
	that substantially underestimates the trace.
	
	Across all the instances, \textsc{XTrace}\xspace produces smaller errors than \HutchPP, sometimes by orders of magnitude.
	For the \texttt{exp} instance (\cref{fig:exp}), the error of \textsc{XTrace}\xspace decays exponentially fast
	at a rate $1.5\times$ faster than \HutchPP.
	The superiority of \textsc{XTrace}\xspace is also visible for the \texttt{step} instance  (\cref{fig:step})
	where \textsc{XTrace}\xspace achieves accuracy $10^{-4}$ with just $m \approx 120$ matvecs, as compared to $m \approx 160$ for \HutchPP.
	
	\NysTrace is frequently the most accurate of the trace estimation methods.
	For the \texttt{exp} instance (\cref{fig:exp}), \NysTrace converges at a rate $2\times$ faster than \textsc{XTrace}\xspace and \NysPP and $3\times$ faster than \HutchPP.
	However, \NysTrace (and \NysPP) can exhibit poor performance for matrices that possess a long tail of slowly decreasing eigenvalues (see the \texttt{step} instance in \cref{fig:step}).
	To understand this phenomenon, observe that the error bounds for \NysTrace
	depend on the trace/nuclear norm, which is sensitive to slow eigenvalue decay (\cref{thm:main}).  We can improve the performance by using the normalization approach of \cref{sec:distribution}, as we detail in the next section.
	
	\subsection{Choice of test vectors}
	\label{sec:test_vectors}
	
	In \cref{sec:distribution}, we recommended an implementation of \textsc{XTrace}\xspace and \NysTrace that uses rotationally invariant test vectors for low-rank approximation and normalizes the distinguished test vector used for trace estimation.
	
	\Cref{fig:test_vectors} shows how this method can improve over estimators that lack the normalization step.
    The figure compares the Girard--Hutchinson estimator, \HutchPP, \textsc{XTrace}\xspace with normalized test vectors, and \textsc{XTrace}\xspace with test vectors from the standard normal distribution $\vec{\omega} \sim \textsc{normal}(\vec{0}, \Id)$, the random sign distribution $\vec{\omega} \sim \textsc{uniform} \{\pm1\}^N$, or the uniform distribution on the sphere $\vec{\omega} \sim \textsc{uniform} \{ \vec{x} \in \real^N : \norm{\vec{x}}_{\ell_2} = \sqrt{N} \} $.
	The differences between \textsc{XTrace}\xspace test vectors are only visible for matrices whose spectra have flat segments, as in the \texttt{flat} and \texttt{step} examples.
	For these examples, the normalization strategy is conspicuously the best, followed by the uniform sign and uniform sphere distributions, with the standard normal distribution lagging well behind.
	
	\begin{figure}[t]
		\centering
		\begin{subfigure}[b]{0.48\textwidth}
			\centering
			\includegraphics[width=\textwidth]{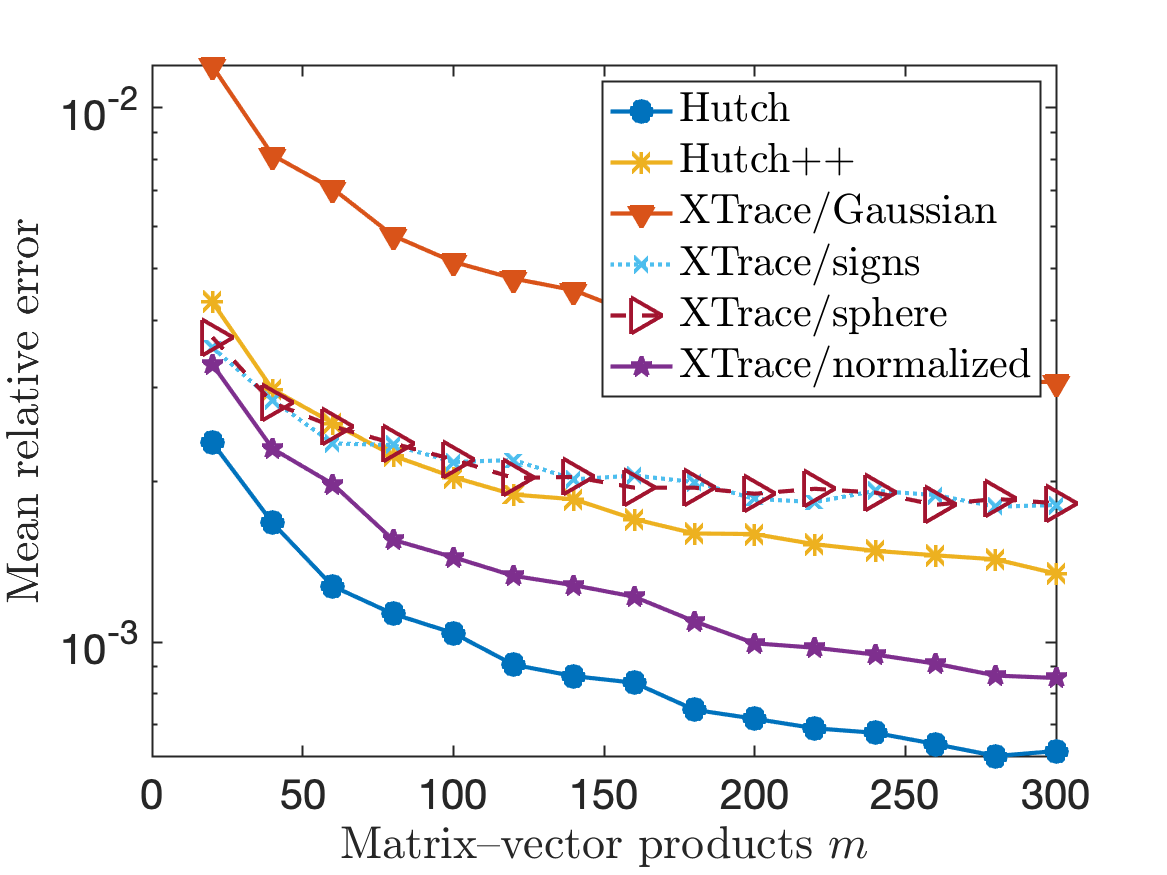}
			\caption{\texttt{flat}} \label{fig:flat_test}
		\end{subfigure}
		~
		\begin{subfigure}[b]{0.48\textwidth}
			\centering
			\includegraphics[width=\textwidth]{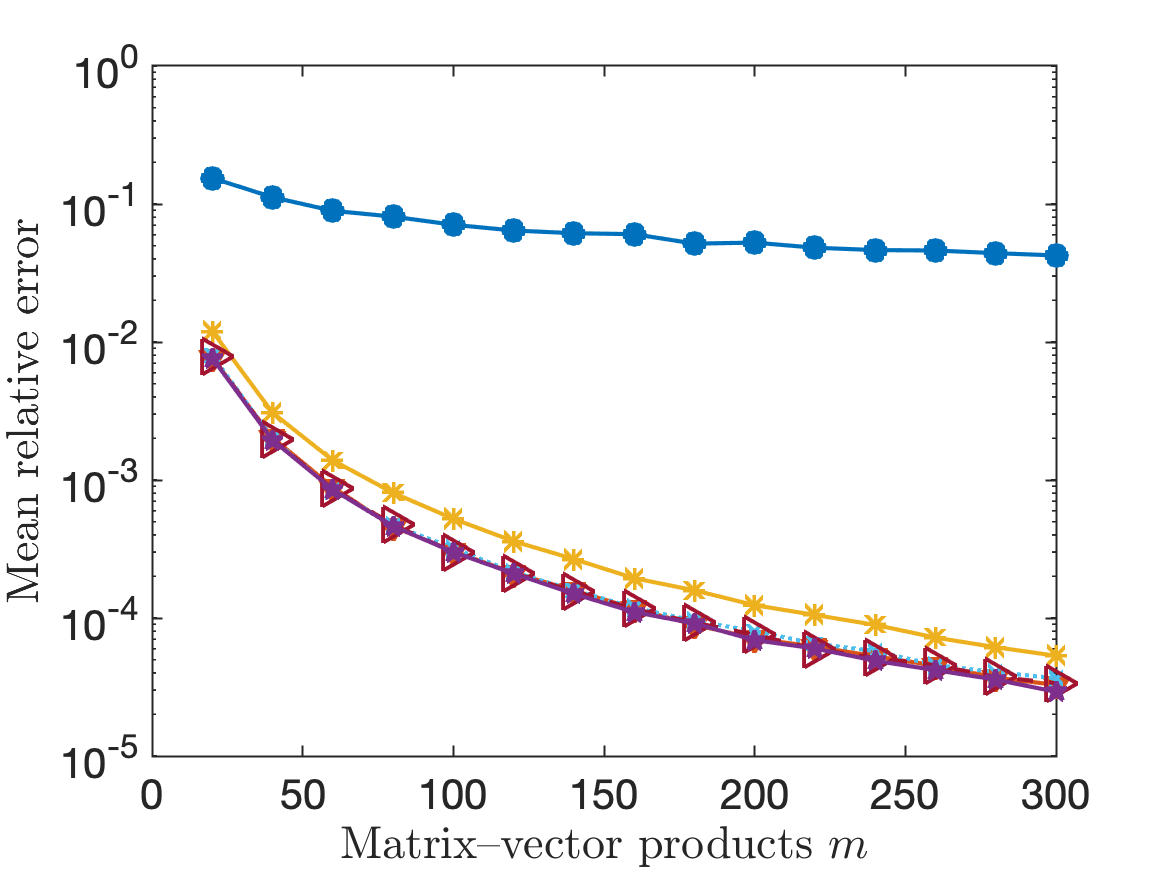}
			\caption{\texttt{poly}} \label{fig:poly_test}
		\end{subfigure}
		
		\begin{subfigure}[b]{0.48\textwidth}
			\centering
			\includegraphics[width=\textwidth]{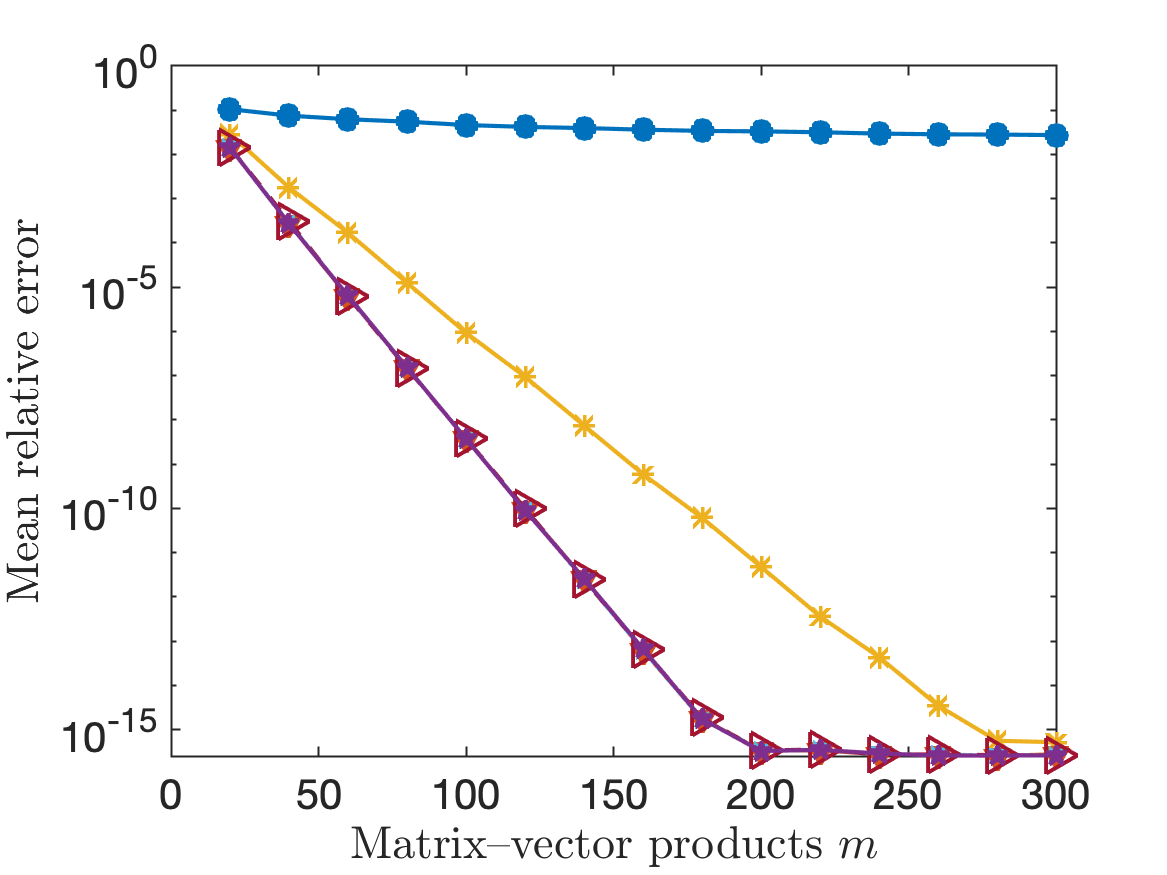}
			\caption{\texttt{exp}} \label{fig:exp_test}
		\end{subfigure}
		\begin{subfigure}[b]{0.48\textwidth}
			\centering
			\includegraphics[width=\textwidth]{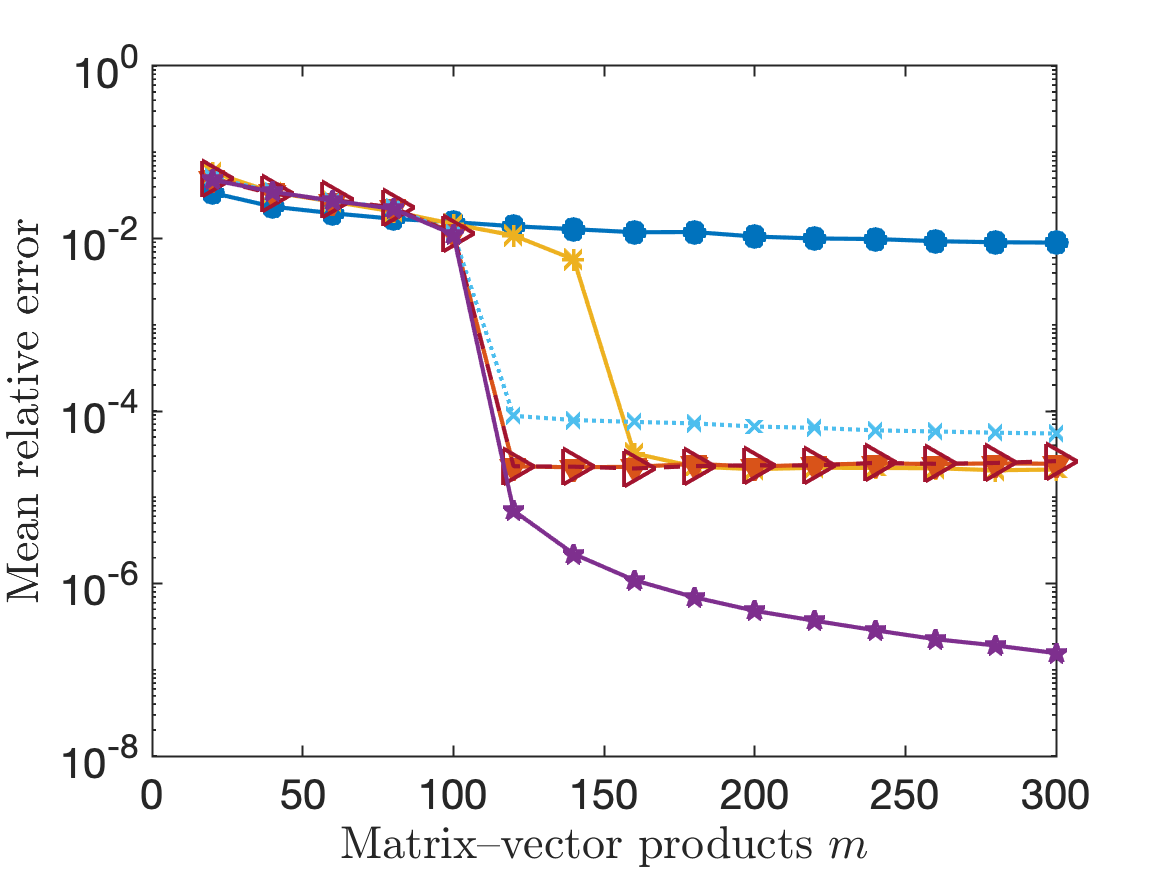}
			\caption{\texttt{step}} \label{fig:step_test}
		\end{subfigure}
		\caption{\textbf{Normalization of test vectors.} The average relative error of the \textsc{XTrace}\xspace estimator using normalized test vectors (\cref{sec:distribution}) as compared with alternative (unnormalized) test vector distributions. See~\cref{sec:test_vectors}.} \label{fig:test_vectors}
	\end{figure}
	
	\subsection{Application: Quantum statistical mechanics}
	\label{sec:quantum}
	
	Our next experiment shows the benefits of using \textsc{XTrace}\xspace and \NysTrace
	for an application in quantum physics.  To compute a
	phase diagram, we must evaluate a large number of trace estimators.
	Our exchangeable estimators
	reduce the number of matvecs needed to achieve a desired tolerance,
	and we can use the posterior error estimator to adaptively
	determine the minimum number $m$ of matvecs.
	
	The average energy of a quantum system with a symmetric Hamiltonian matrix $\mat{H}\in\real^{N\times N}$ at inverse temperature $\beta > 0$ is
	\begin{equation*} \label{eq:average_energy}
		E(\beta) = \frac{1}{Z(\beta)} \tr [\mat{H}\exp(-\beta \mat{H})] \quad \text{where $Z(\beta) = \tr \exp(-\beta \mat{H})$}.
	\end{equation*}
	The quantity $Z(\beta)$ is the partition function, introduced in \cref{sec:numerical_experiments}.
	We observe that $\tr [\mat{H}\exp(-\beta \mat{H})]$ and $\tr\exp(-\beta\mat{H})$ are ideal candidates for estimation using \textsc{XTrace}\xspace and \NysTrace, since the matrix exponential leads to rapidly decaying eigenvalues.
	
	We apply \textsc{XTrace}\xspace and \NysTrace to compute the partition function and energy for the transverse field Ising model (TFIM) for a periodic 1D chain \cite{Pfe70}, which is specified by the Hamiltonian matrix
	\begin{equation} \label{eq:tfim}
		\mat{H} = -\sum_{i=1}^n \mat{Z}_i \mat{Z}_{i+1} - h \sum_{i=1}^n \mat{X}_i \in \real^{2^n\times 2^n}.
	\end{equation}
	Here, $\mat{X}_i$ and $\mat{Z}_i$ denote Pauli operators acting on the $i$th site; that is, 
	\begin{align*}
		\mat{X}_i = \Id_{2\times 2}^{\otimes (n-i-1)} \otimes \twobytwo{0}{1}{1}{0} \otimes \Id_{2\times 2}^{\otimes (n-i)},\quad \mat{Z}_i = \Id_{2\times 2}^{\otimes (n-i-1)} \otimes \twobytwo{1}{0}{0}{-1} \otimes \Id_{2\times 2}^{\otimes (n-i)},
	\end{align*}
	and $\mat{Z}_{n+1} = \mat{Z}_1$ by periodicity.
	The eigenvalues of $\mat{H}$ are known exactly \cite[eqs.~(16)--(17)]{Lit19}, which allows us to precisely evaluate the error of stochastic estimates of $Z(\beta) = Z(\beta, h)$ and $E(\beta) = E(\beta, h)$.
	Before applying stochastic trace estimation, we shift the Hamiltonian matrix by a constant $b = (1 + h)n$ so that $\mat{H} + b \Id$ is positive semidefinite.
	
	\begin{figure}[t]
		\centering
		\begin{subfigure}{0.39\textwidth}
			\includegraphics[width=0.99\textwidth]{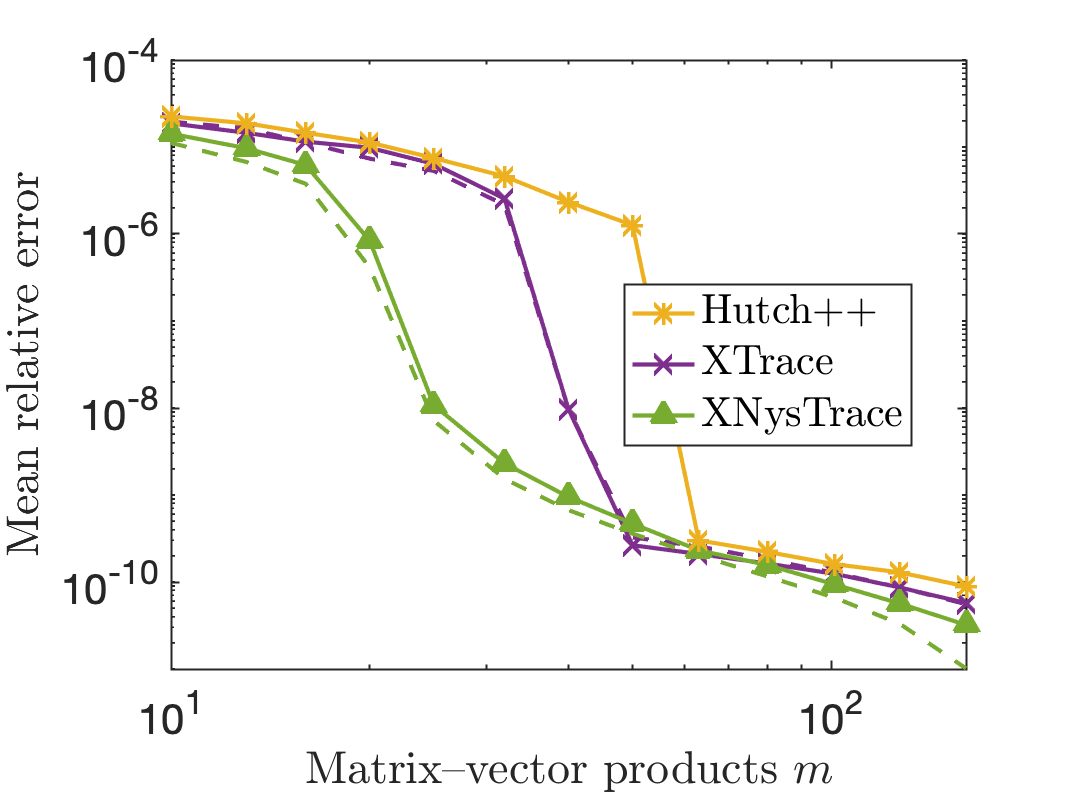}
			\caption{}\label{fig:tfim_errs}
		\end{subfigure}
		~
		\begin{subfigure}{0.58\textwidth}
			\includegraphics[width=0.99\textwidth]{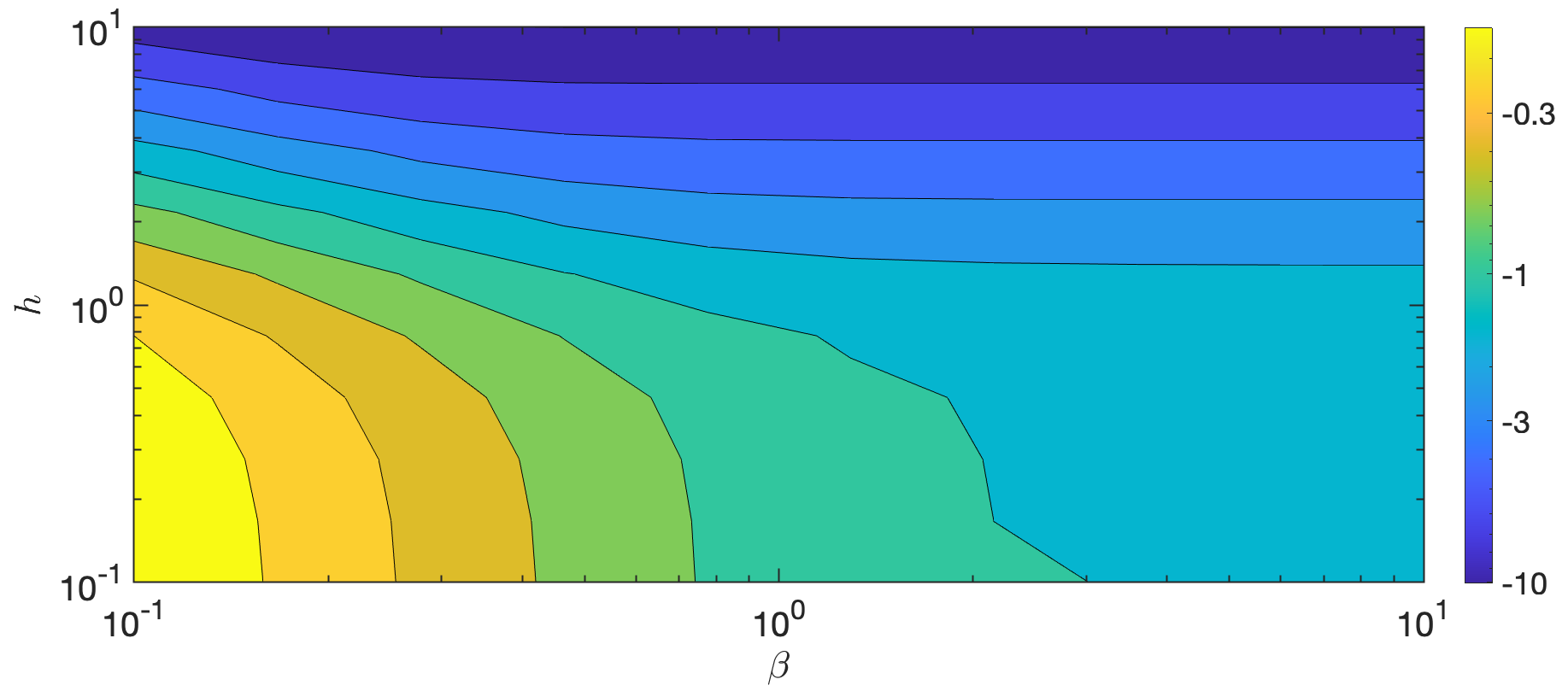}
			\caption{}\label{fig:tfim_energies}
		\end{subfigure}
		\caption{\textbf{Quantum statistical mechanics.}  \textit{Left}:
			Mean relative error (solid lines) and mean posterior error estimates (dashed lines) when computing the partition function $Z(\beta = 0.6, h = 10)$.
			\textit{Right}: Average energy $E(\beta,h)/n$ per site for $\beta,h \in [10^{-1},10^1]$ as computed by \NysTrace. See~\cref{sec:quantum}.}
		\label{fig:tfim}
	\end{figure}
	
	\Cref{fig:tfim_errs} shows the errors of \HutchPP, \textsc{XTrace}\xspace, and \NysTrace when computing the partition function $Z(\beta, h)$ of the TFIM with $n = 18$, $h = 10$, and $\beta = 0.6$; this is the same setting as in \cref{fig:tfim_single_shot}.
    (The Girard--Hutchinson estimator is omitted because the error is four orders of magnitude larger than the other methods.)
	The thick lines indicate the average errors over 10 trials, while the dashed lines (for \textsc{XTrace}\xspace and \NysTrace) indicate the average error estimates introduced in \cref{sec:error}.
	We observe that the error estimates closely track the true errors, differing by a factor of at most $3.2$.
	
	\Cref{fig:tfim_energies} shows the average energy $E(\beta,h)/n$ per site for parameters $\beta,h \in [10^{-1},10^1]$, up to a relative error of $10^{-3}$.
	We compute the energy by using \NysTrace, together with the doubling strategy from \cref{sec:adaptive_stopping}.
	To ensure the robustness of the doubling strategy, we use a slightly stricter tolerance $\varepsilon = 10^{-4}$ than our desired accuracy of $10^{-3}$.
	
	\subsection{Application: Networks}
	\label{sec:networks}
	
	One of the basic problems in network science is to measure the centrality of each node in a graph.
	We focus on two centrality measures, which can be defined in terms of the adjacency matrix $\mat{M}$:
	
	\begin{enumerate}
		\item The \textit{number of triangles} \cite{AD18} incident on node $i$ is given by $\triangle_i(\mat{M}) = \tfrac{1}{2} (\mat{M}^3)_{ii}$.
		\item The \textit{subgraph centrality} \cite{Est22} of node $i$ is defined as $\mathrm{SC}_i \coloneqq (\exp(\mat{M}))_{ii}$.
	\end{enumerate}
	Both centrality measures are the diagonal entries of functions of the adjacency matrix.
	
	Using the BKS, \DiagPP, and \XDiag diagonal estimators, we estimate these centrality measures for the protein--protein interaction network $\mat{M} \in \real^{2361\times 2361}$ for budding yeast \cite{BZC+03}, available in the SuiteSparse collection \cite{DH11}.
	(Following \cite{BN22}, \DiagPP is omitted for the triangle problem because $\mat{M}^3/2$ is not psd.)
    The matrix $\mat{M}^3/2$ has modest spectral decay ($\sigma_1/\sigma_{100} = 40$) and the matrix $\exp(\mat{M})$ has significant spectral decay ($\sigma_1/\sigma_{100} = 10^6$).
	We evaluate the quality of our estimates using the relative $\ell_\infty$ error
	\begin{equation*}
		\operatorname{error}(\vec{\hat{\diag}}) \coloneqq \frac{\max_{1\le i \le N}|a_{ii} - \hat{\diag}_i|}{\max_{1\le i \le N}|a_{ii}|},
	\end{equation*}
	averaged over 1000 trials.
	\Cref{fig:network_diag} shows the results.
	For the subgraph centrality problem
	after $m = 200$ matvecs, \XDiag is more accurate than BKS by five orders of magnitude and more accurate than \DiagPP by three orders of magnitude.
	
	\begin{figure}[t]
		\centering
		\begin{subfigure}[b]{0.48\textwidth}
			\centering
			\includegraphics[width=\textwidth]{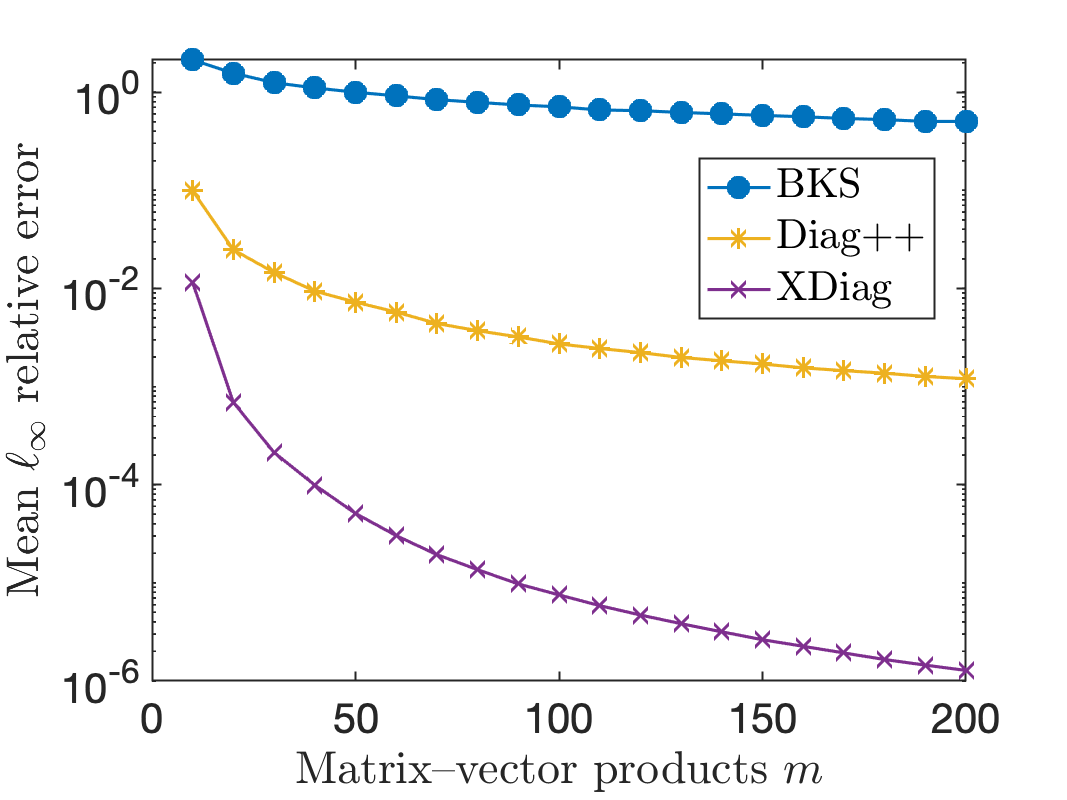}
			\caption{Subgraph centralities $\operatorname{SC}_i(\mat{M})$} \label{fig:estrada_diag}
		\end{subfigure}
		~
		\begin{subfigure}[b]{0.48\textwidth}
			\centering
			\includegraphics[width=\textwidth]{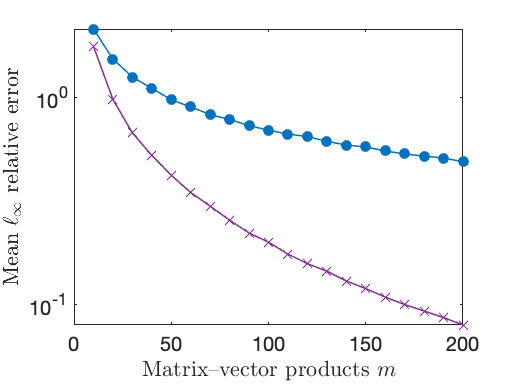}
			\caption{Numbers of triangles $\triangle_i(\mat{M})$} \label{fig:triangles_diag}
		\end{subfigure}
		\caption{\textbf{Networks.}  Mean error of BKS diagonal estimator and \XDiag for the subgraph centralities (\emph{left}) and triangle numbers (\emph{right}) for the \texttt{yeast} graph.  See~\cref{sec:networks}.} \label{fig:network_diag}
	\end{figure}
	
	\section{Theoretical analysis} 
	\label{sec:proofs}
	
	In this final section, we prove \cref{thm:main}, which provides refined error bounds for three trace estimators, and we prove \cref{prop:estim_bound}, which describes the behavior of the posterior error estimator.
	
	\subsection{\textsc{XTrace}\xspace variance bound} \label{sec:general_error}
	
	To begin, let us establish an initial variance bound for the \textsc{XTrace}\xspace estimator.
	This result shows that the variance depends on the error in a low-rank approximation
	of the input matrix.  Later, we will bound these errors using standard results
	for the randomized SVD.
	
	\begin{proposition}[\textsc{XTrace}\xspace error] \label{prop:x}
		Fix $\mat{A} \in \real^{N \times N}$, and consider the \textsc{XTrace}\xspace estimator $\hat{\tr}_{\rm X}$ defined in~\cref{eq:xtrace} with a test matrix $\mat{\Omega} = \begin{bmatrix} \vec{\omega}_1 & \dots & \vec{\omega}_{m/2} \end{bmatrix}$ consisting of $m/2$ standard normal test vectors.
		The estimator is unbiased: $\expect \hat{\tr}_{\rm X} = \tr \mat{A}$.  Moreover, the variance satisfies
		\begin{align*}
			\Var[\hat{\tr}_{\rm X}]
			= \mathbb{E}\, \bigl| \hat{\tr}_{\rm X} - \tr \mat{A} \bigr|^2
			&\leq \frac{2}{m/2} \mathbb{E} \bigl \lVert \big(\Id-\mat{Q}_{(1)} \mat{Q}_{(1)}^*\big) \mat{A} \rVert^2_{\rm F}
			+ 4\, \mathbb{E}\, \bigl\lVert \big(\Id-\mat{Q}_{(12)} \mat{Q}_{(12)}^*\big) \mat{A} \bigr\rVert^2,
		\end{align*}
		where $\mat{Q}_{(i)} = \orth(\mat{A}\mat{\Omega}_{-i})$ and $\mat{Q}_{(ij)} = \orth(\mat{A}\mat{\Omega}_{-ij})$.
	\end{proposition}
	
	\begin{proof}
		For all indices $1 \leq i, j \leq m/2$, we abbreviate the orthogonal projectors $\mat{\Pi}_i \coloneqq \mat{Q}_{(i)}^{\phantom{*}} \mat{Q}_{(i)}^*$
		and $\mat{\Pi}_{ij} \coloneqq  \mat{Q}_{(ij)}^{\phantom{*}} \mat{Q}_{(ij)}^*$.  Note that $\vec{\omega}_i$ is independent from $\mat{\Pi}_i$, while $(\vec{\omega}_i, \vec{\omega}_j)$ is independent from $\mat{\Pi}_{ij}$.
		
		For each $1 \leq i \leq m/2$, we can check that the basic estimator $\hat{\tr}_i$ defined in~\cref{eq:trX-i} is unbiased.
		To do so, condition on $\mat{\Omega}_{-i}$ and average over $\vec{\omega}_i$ to arrive at the identity
		\begin{equation} \label{eq:xtrace_unbiased}
			\begin{split}
				\mathbb{E} \bigl[\hat{\tr}_i \,\big|\, \mat{\Omega}_{-i} \bigr] 
				&= \mathbb{E} \bigl[\tr (\mat{\Pi}_i\mat{A}\mat{\Pi}_i ) + \vec{\omega}_i^*(\Id-\mat{\Pi}_i) \mat{A} (\Id-\mat{\Pi}_i )\vec{\omega}_i \,\big|\, \mat{\Omega}_{-i} \bigr] \\
				&= \mathbb{E} \bigl[\tr(\mat{\Pi}_i\mat{A}\mat{\Pi}_i ) + \tr \big( (\Id-\mat{\Pi}_i) \mat{A} (\Id-\mat{\Pi}_i ) \big) \,\big|\, \mat{\Omega}_{-i} \bigr]
				= \tr \mat{A}.
			\end{split}
		\end{equation}
		The second equality uses the fact that each test vector $\vec{\omega}_i$ is isotropic 
		and independent from $\mat{\Pi}_i$, which is a function of $\mat{\Omega}_{-i}$.  The third equality follows when we cycle the traces
		and invoke the fact that the projector $\Id - \mat{\Pi}_i$ is idempotent.
		We confirm that $\hat{\tr}_i$ is unbiased by applying the tower property to take the total expectation of~\cref{eq:xtrace_unbiased}.
		The full estimator $\hat{\tr}_{\rm X}$ is unbiased because it is an average of the unbiased estimators $\hat{\tr}_i$.
		
		Next, to bound the variance, we use the exchangeability of $\vec{\omega}_1, \ldots, \vec{\omega}_{m/2}$ to compute
		\begin{equation*}
			\Var \bigl[\hat{\tr}_{\rm X}\bigr]
			= \Var \Biggl[\frac{1}{m/2} \sum_{i=1}^{m/2} \hat{\tr}_i\Biggr]
			= \Biggl(\frac{1}{m/2}\Biggr)\, \underbrace{\Var\bigl[\hat{\tr}_1\bigr]}_{\sf A} 
			+ \Biggl(1 - \frac{1}{m/2}\Biggr)\, \underbrace{\Cov\bigl[\hat{\tr}_1, \hat{\tr}_2\bigr]}_{\sf B}.
		\end{equation*}
		Hence, $\Var \bigl[\hat{\tr}_{\rm X}\bigr]$ is the weighted average of a variance term {\sf A} and a covariance term {\sf B}.
		
		To evaluate the variance term {\sf A}, condition on $\mat{\Omega}_{-1}$
		and average over $\vec{\omega}_1$.  Thus,
		\begin{align*}
			{\sf A} &= \Var\bigl[\hat{\tr}_1\bigr]
			= \expect \bigl[ \Var\big[ \hat{\tr}_1 \, \big\vert \, \mat{\Omega}_{-1} \big] \bigr]
			+ \Var\bigl[ \expect\bigl[ \hat{\tr}_1 \, \big\vert \, \mat{\Omega}_{-1} \bigr] \bigr] \\
			&= \expect \bigl[ \Var \bigl[ \tr(\mat{\Pi}_1^*\mat{A}\mat{\Pi}_{1} \big) + \vec{\omega}_1^* (\Id-\mat{\Pi}_1) \mat{A} (\Id-\mat{\Pi}_i)  \vec{\omega}_1 \, \big\vert \, \mat{\Omega}_{-1} \bigr] \bigr] \\
			&= \expect \bigl[ \Var\bigl[
			\vec{\omega}_1^* (\Id-\mat{\Pi}_{1}) \mat{A} (\Id-\mat{\Pi}_{1}) \vec{\omega}_1
			\,\big|\, \mat{\Omega}_{-1}
			\bigr] \bigr] \leq 2\,\mathbb{E} \bigl \lVert (\Id-\mat{\Pi}_{1}) \mat{A} (\Id-\mat{\Pi}_{1}) \bigr \rVert^2_{\rm F}.
		\end{align*}
		The first relation is the chain rule for the variance.
		To pass to the second line, we invoke the fact \cref{eq:xtrace_unbiased} that the conditional expectation is constant. To pass to the third line, we drop the trace, which is conditionally constant.  The last relation follows from a direct calculation using the facts that $\vec{\omega}_1$ is standard normal
		and independent from $\mat{\Pi}_{1}$. 
		
		To bound the covariance term {\sf B}, it is helpful to isolate the part of the covariance
		that only depends on $\vec{\omega}_1$ and $\vec{\omega}_2$.  We rely on the following
		observation.  For any (random) matrix $\mat{X} \in \mathbb{R}^{N \times N}$
		that is independent from $\vec{\omega}_1$ and $\vec{\omega}_2$, we may calculate
		that
		\begin{align}
			\label{eq:start}
			& \mathbb{E}\bigl[\bigl(\hat{\tr}_1 - \tr \mat{A} + \tr \mat{X} - \vec{\omega}_1^* \mat{X} \vec{\omega}_1^{\vphantom{*}}
			\bigr)
			\bigl(\hat{\tr}_2 - \tr \mat{A} + \tr \mat{X} - \vec{\omega}_2^* \mat{X} \vec{\omega}_2^{\vphantom{*}}
			\bigr)\bigr] \\ 
			\label{eq:mid}
			&\qquad= \mathbb{E}\bigl[\bigl(\hat{\tr}_1 - \tr \mat{A} + \tr \mat{X} - \vec{\omega}_1^* \mat{X} \vec{\omega}_1^{\vphantom{*}}
			\bigr)
			\bigl(\hat{\tr}_2 - \tr \mat{A}\bigr)\bigr] \\
			\label{eq:end}
			&\qquad= \mathbb{E}\bigl[\bigl(\hat{\tr}_1 - \tr \mat{A} \bigr)
			\bigl(\hat{\tr}_2 - \tr \mat{A} \bigr)\bigr]
			= \Cov\bigl[ \hat{\tr}_1, \hat{\tr}_2 \bigr].
		\end{align}
		To pass to \cref{eq:mid}, we condition on $\mat{\Omega}_{-1}$ and average over $\vec{\omega}_1$, exploiting the fact~\cref{eq:xtrace_unbiased} that $\hat{\tr}_1$ is an unbiased estimator of $\tr \mat{A}$, conditional on $\mat{\Omega}_{-1}$.
		To pass to \cref{eq:end}, condition on $\mat{\Omega}_{-2}$ and average over $\vec{\omega}_2$.
		
		To continue, select the particular random matrix $\mat{X} = (\Id-\mat{\Pi}_{12}) \mat{A} (\Id-\mat{\Pi}_{12})$.  Applying the Cauchy--Schwarz inequality, we find that
		\begin{align*}
			{\sf B} = \Cov\bigl[\hat{\tr}_1, \hat{\tr}_2\bigr]
			&= \mathbb{E}\bigl[\bigl(\hat{\tr}_1 - \tr \mat{A} + \tr \mat{X} - \vec{\omega}_1^* \mat{X} \vec{\omega}_1
			\bigr)
			\bigl(\hat{\tr}_2 - \tr \mat{A} + \tr \mat{X} - \vec{\omega}_2^* \mat{X} \vec{\omega}_2
			\bigr)\bigr] \\
			&\leq \mathbb{E}\bigl|\hat{\tr}_1 - \tr \mat{A} + \tr \mat{X} - \vec{\omega}_1^* \mat{X} \vec{\omega}_1
			\bigr|^2 \\
			&= \mathbb{E}\Var\bigl[
			\vec{\omega}_1^*\bigl[(\Id-\mat{\Pi}_{1}) \mat{A} (\Id-\mat{\Pi}_{1}) - \mat{X}\bigr]\vec{\omega}_1
			\,\big|\, \vec{\Omega}_{-1}
			\bigr] \\
			&\leq 2 \,\mathbb{E} \bigl\lVert (\Id-\mat{\Pi}_{1}) \mat{A} (\Id-\mat{\Pi}_{1}) - \mat{X} \bigr\rVert_{\rm F}^2.
		\end{align*}
		Since $\operatorname{range}(\mat{Q}_{(12)}) \subseteq \operatorname{range}(\mat{Q}_{(1)})$, we have the relations $\mat{\Pi}_{12} = \mat{\Pi}_1 \mat{\Pi}_{12}$ and $\Id - \mat{\Pi}_1 = (\Id - \mat{\Pi}_1)(\Id - \mat{\Pi}_{12})$.
		Since $\mat{X} = (\Id - \mat{\Pi}_{12}) \mat{A} (\Id - \mat{\Pi}_{12})$,
		it follows that
		\begin{align*}
			\bigl\lVert \big(\Id-\mat{\Pi}_1\big) \mat{A}\big(\Id-\mat{\Pi}_1\big) - \mat{X} \bigr\rVert_{\rm F}^2 
			&= \bigl\lVert \big(\Id-\mat{\Pi}_1\big) \mat{X}\big(\Id-\mat{\Pi}_1\big) - \mat{X} \bigr\rVert_{\rm F}^2 \\
			&= \bigl\lVert \big(\Id-\mat{\Pi}_1\big) \mat{X} \mat{\Pi}_1 + \mat{\Pi}_1 \mat{X} \bigr\rVert_{\rm F}^2 \\
			&= \bigl\lVert \big(\Id-\mat{\Pi}_1\big) \mat{X} \mat{\Pi}_1 \bigr\rVert_{\rm F}^2 
			+ \bigl\lVert \mat{\Pi}_1 \mat{X} \bigr\rVert_{\rm F}^2\\
			&\leq \norm{ \mat{X} \mat{\Pi}_1 }_{\rm F}^2 + \norm{  \mat{\Pi}_1 \mat{X} }_{\rm F}^2 \\
			&= \bigl\lVert \mat{X} \bigl(\mat{\Pi}_1 - \mat{\Pi}_{12}\bigr) \bigr\rVert_{\rm F}^2
			+ \bigl\lVert \bigl(\mat{\Pi}_1 - \mat{\Pi}_{12}\bigr) \mat{X} \bigr\rVert_{\rm F}^2 \\
			&= \bigl\lVert \mat{X} \bigl(\mat{\Pi}_1 - \mat{\Pi}_{12}\bigr) \bigr\rVert^2
			+ \bigl\lVert \bigl(\mat{\Pi}_1 - \mat{\Pi}_{12}\bigr) \mat{X} \bigr\rVert^2
			\leq 2 \norm{ \mat{X} }^2.
		\end{align*}
		We invoke the Pythagorean theorem to pass to the third line.  To reach the fifth line,
		exploit the representations $\mat{X} = \mat{X} (\Id - \mat{\Pi}_{12})$ and 
		$\mat{X} = (\Id - \mat{\Pi}_{12}) \mat{X}$.
		To pass to the sixth line, note that $\mat{\Pi}_1 - \mat{\Pi}_{12}$ is a rank-one orthogonal projector;
		the Frobenius norm and spectral norm coincide for rank-one matrices.
		Combining the last two displays, we deduce that
		\begin{equation*}
			{\sf B} = \Cov\bigl[\hat{\tr}_1, \hat{\tr}_2\bigr] \leq 4 \,\mathbb{E}\,\bigl\lVert (\Id-\mat{\Pi}_{12}) \mat{A} (\Id-\mat{\Pi}_{12}) \bigr\rVert^2.
		\end{equation*}
		Combining the estimates for {\sf A} and {\sf B}, we achieve the stated bound for the variance.
	\end{proof}
	
	\subsection{\HutchPP variance bound}
	
	By a similar argument, we can obtain an initial variance bound for the \HutchPP estimator.
	This result is more elementary because it does not require us to account for
	interactions between the simple estimators.
	
	\begin{proposition}[\HutchPP error] \label{prop:hpp}
		Fix $\mat{A} \in \real^{N \times N}$\!, and consider the \HutchPP estimator $\hat{\tr}_{\rm H++}$ defined in~\cref{eq:hutch++}
		with $2m/3$ standard normal test vectors.  The estimator is unbiased: $\expect \hat{\tr}_{\rm H++} = \tr \mat{A}$.
		Moreover, the variance satisfies
		\begin{align*}
			\mathbb{E} \,\bigl| \hat{\tr}_{\rm H++} - \tr \mat{A} \bigr|^2
			&\leq \frac{2}{m/3} \mathbb{E} \lVert (\Id - \mat{Q} \mat{Q}^\ast) \mat{A} \rVert_{\rm F}^2,
		\end{align*}
		where $\mat{Q} = \orth(\mat{A}\mat{\Omega})$ and $\mat{\Omega} = \begin{bmatrix} \vec{\omega}_{m/3+1} & \dots & \vec{\omega}_{2m/3} \end{bmatrix}$.
	\end{proposition}
	
	\begin{proof}
		The idea is to condition on the low-rank approximation $\mat{QQ}^* \mat{A}$
		and invoke the chain rule for the variance, as in the proof of \cref{prop:x}.
		See the argument in~\cite[Thm.~10]{MMMW21a}, which was supplied by the second author of this paper.
	\end{proof}
	
	\subsection{\NysTrace variance bound}
	
	Last, we establish an initial variance bound for the \NysTrace estimator.
	This result shows how the variance depends on the error in a randomized
	Nystr{\"o}m approximation.  Later, we will use recent results for the
	Nystr{\"o}m approximation to obtain a complete variance bound.
	
	\begin{proposition}[\NysTrace error] \label{prop:xn}
		Let $\mat{A} \in \real^{N \times N}$ be psd.
		Consider the \NysTrace estimator $\hat{\tr}_{NX}$ with $m$ standard normal test vectors
		as defined in~\cref{eq:xnystrace}.  The estimator is unbiased: $\expect \hat{\tr}_{\rm NX} = \tr \mat{A}$.
		Moreover, the variance satisfies the bound
		\begin{align*}
			\mathbb{E}\, \bigl| \hat{\tr}_{\rm XN} - \tr \mat{A} \bigr|^2
			&\leq \frac{2}{m} \mathbb{E} \bigl\lVert \mat{A} - \mat{A}\langle \mat{\Omega}_{-1}\rangle \bigr\rVert_{\rm F}^2 
			+ 2\,
			\mathbb{E} \bigl\lVert \mat{A} - \mat{A}\langle \mat{\Omega}_{-12}\rangle \bigr\rVert^2.
		\end{align*}
	\end{proposition}
	\begin{proof}
		The proof resembles the proof of \cref{prop:x} but is slightly simpler.
		The unbiasedness of \NysTrace follows from a short computation similar to \cref{eq:xtrace_unbiased}.
		
		To control the variance, we calculate that
		\begin{equation*}
			\Var \bigl[\hat{\tr}_{\rm XN}\bigr]
			= \Var \Biggl[\frac{1}{m} \sum_{i=1}^{m} \hat{\tr}_i\Biggr]
			= \Biggl(\frac{1}{m}\Biggr)\, \underbrace{\Var\bigl[\hat{\tr}_1\bigr]}_{\sf A} 
			+ \Biggl(1 - \frac{1}{m}\Biggr)\, \underbrace{\Cov\bigl[\hat{\tr}_1, \hat{\tr}_2\bigr]}_{\sf B}.
		\end{equation*}
		The variance term is exactly
		$$
		{\sf A} = \Var\bigl[\hat{\tr}_1\bigr] = 2\mathbb{E} \bigl \lVert \mat{A} - \mat{A}\langle \mat{\Omega}_{-1} \rangle \bigr \rVert^2_{\rm F}.
		$$
		To bound the covariance term {\sf B}, we set
		$\mat{X} = \mat{A} - \mat{A}\langle \mat{\Omega}_{-12} \rangle$ in~\cref{eq:start}.
		Applying the Cauchy--Schwarz inequality, we find that
		\begin{align*}
			{\sf B} = \Cov\bigl[\hat{\tr}_1, \hat{\tr}_2\bigr]
			&= \mathbb{E}\bigl[\bigl(\hat{\tr}_1 - \tr \mat{A} + \tr \mat{X} - \vec{\omega}_1^* \mat{X} \vec{\omega}_1
			\bigr)
			\bigl(\hat{\tr}_2 - \tr \mat{A} + \tr \mat{X} - \vec{\omega}_2^* \mat{X} \vec{\omega}_2
			\bigr)\bigr] \\
			&\leq \mathbb{E}\bigl|\hat{\tr}_1 - \tr \mat{A} + \tr \mat{X} - \vec{\omega}_1^* \mat{X} \vec{\omega}_1
			\bigr|^2 \\
			&= \mathbb{E}\Var\bigl[
			\vec{\omega}_1^*\bigl[\mat{A}\langle \mat{\Omega}_{-12} \rangle
			- \mat{A}\langle \mat{\Omega}_{-1} \rangle\bigr]\vec{\omega}_1
			\,\big|\, \mat{\Omega}_{-1}
			\bigr] \\
			&= 2 \,\mathbb{E} \bigl\lVert \mat{A}\langle \mat{\Omega}_{-1} \rangle
			- \mat{A}\langle \mat{\Omega}_{-12} \rangle \bigr\rVert_{\rm F}^2.
		\end{align*}
		The psd matrix $\mat{A}\langle \mat{\Omega}_{-1} \rangle - \mat{A}\langle \mat{\Omega}_{-12} \rangle$ has rank one,
		and it is bounded above by $\mat{A} - \mat{A}\langle \mat{\Omega}_{-12} \rangle$ in the psd order.  Therefore,
		\begin{equation*}
			{\sf B} = \Cov\bigl[\hat{\tr}_1, \hat{\tr}_2\bigr]
			\leq 2 \, \mathbb{E} \bigl\lVert \mat{A} - \mat{A}\langle \mat{\Omega}_{-12}\rangle \bigr\rVert^2.
		\end{equation*}
		Combine the displays to complete the proof.
	\end{proof}
	
	\subsection{Error bounds for low-rank approximations}
	
	To prove the main result, \cref{thm:main}, we need two auxiliary lemmas.
	First, we present error bounds for randomized SVD and randomized Nystr\"om approximation,
	drawn from the recent paper \cite[Thm.~6.7 and Cor.~6.8]{TW23a}.
	
	\begin{lemma}[Randomized SVD and randomized Nystr\"om error] \label{lem:bounds}
		Fix a matrix $\mat{A} \in \mathbb{R}^{N \times N}$, and draw a standard normal matrix
		$\mat{\Omega} \in \mathbb{R}^{N \times k}$.
		For any $r \leq k - 2$, the randomized SVD error is bounded by
		\begin{align*}
			& \mathbb{E} \lVert (\Id - \mat{Q}\mat{Q}^*)\mat{A}  \rVert^2 
			\leq \frac{k + r - 1}{k - r - 1} \Biggl(\lVert \mat{A} - \lowrank{\mat{A}}_r \rVert^2
			+ \frac{\e^2}{k-r} \lVert \mat{A} - \lowrank{\mat{A}}_r \rVert^2_{\rm F} \Biggr), \\
			& \mathbb{E} \lVert (\Id - \mat{Q}\mat{Q}^*)\mat{A} \rVert^2_{\rm F} 
			\leq \frac{k - 1}{k - r - 1} 
			\lVert \mat{A} - \lowrank{\mat{A}}_r \rVert^2_{\rm F},
		\end{align*}
		where $\mat{Q} = \operatorname{orth}(\mat{A}\mat{\Omega})$.
		
		Assume that $\mat{A \in \real^{N \times N}}$ is a psd matrix.  For any $r \leq k - 4$, the randomized Nystr\"om error is bounded by
		\begin{align*}
			& \bigl(\mathbb{E} \lVert \mat{A} - \mat{A}\langle \mat{\Omega}\rangle \rVert^2\bigr)^{1 \slash 2} 
			\leq \frac{k + r - 1}{k - r - 3} \Biggl(\lVert \mat{A} - \lowrank{\mat{A}}_r \rVert
			+ \frac{\sqrt{3} \e^2}{k - r} \lVert \mat{A} - \lowrank{\mat{A}}_r \rVert_{\ast} \Biggr), \\
			& \bigl(\mathbb{E} \lVert \mat{A} - \mat{A}\langle \mat{\Omega}\rangle \rVert^2_{\rm F}\bigr)^{1 \slash 2} 
			\leq \frac{k - 2}{k - r - 3} \Biggl(\lVert \mat{A} - \lowrank{\mat{A}}_r \rVert_{\rm F} + \frac{1}{\sqrt{k - r}} \lVert \mat{A} - \lowrank{\mat{A}}_r \rVert_{\ast} \Biggr).
		\end{align*}
	\end{lemma}
	
	Second, we report a standard fact about the decay rate of the singular values, which is also exploited in \cite[Lem.~13]{MMMW21a} and \cite[Lem.~7]{GSTV07}.  We omit the easy proof.
	
	\begin{fact} \label{lem:trace_relations}
		For any matrix $\mat{A} \in \mathbb{R}^{N \times N}$ and any $r \geq 1$,
		\begin{equation*}
			\lVert \mat{A} - \lowrank{\mat{A}}_r\rVert \leq \frac{\lVert \mat{A} \rVert_\ast}{r + 1},
			\qquad
			\lVert \mat{A} - \lowrank{\mat{A}}_r\rVert_{\rm F} \leq \frac{\lVert \mat{A} \rVert_\ast}{2\sqrt{r}}.
		\end{equation*}
	\end{fact}
	
	\subsection{The complete variance bound} \label{sec:main_proof}
	
	To establish the main result, \cref{thm:main}, we begin with the initial variance
	bounds and introduce the results from \cref{lem:bounds,lem:trace_relations}.
	
	\begin{proof}[Proof of \cref{thm:main}]
		We recognize that all the terms in the error formulas in \cref{prop:hpp,prop:x,prop:xn}
		reflect the squared approximation error in a randomized SVD or a randomized Nystr\"om approximation.
		Therefore, we can apply the error bounds in \cref{lem:bounds} to obtain more explicit error representations.
		For \HutchPP, when $r \leq m/3 -2$,
		\begin{equation*}
			\mathbb{E} \,\bigl| \hat{\tr}_{\rm H++} - \tr \mat{A} \bigr|^2
			\leq \frac{2}{m/3} \mathbb{E} \lVert (\Id - \mat{Q} \mat{Q}^\ast) \mat{A} \rVert_{\rm F}^2
			\leq \frac{2}{m/3 - r - 1} \lVert \mat{A} - \lowrank{\mat{A}}_r \rVert_{\rm F}^2.
		\end{equation*}
		For \textsc{XTrace}\xspace, when $r \leq m/2-4$,
		\begin{align*}
			\mathbb{E}\, \bigl| \hat{\tr}_{\rm X} - \tr \mat{A} \bigr|^2
			&\leq \frac{2}{m/2} \mathbb{E} \bigl \lVert \big(\Id-\mat{Q}_{(1)} \mat{Q}_{(1)}^*\big) \mat{A} \rVert^2_{\rm F}
			+ 4 \mathbb{E}\, \bigl\lVert \big(\Id-\mat{Q}_{(12)} \mat{Q}_{(12)}^*\big) \mat{A} \bigr\rVert^2 \\
			&\leq \frac{4m}{m/2 - r - 3} 
			\lVert \mat{A} - \lowrank{\mat{A}}_r \rVert^2
			+ \frac{4\e^2 m}{(m/2 - r - 3)^2} \lVert \mat{A} - \lowrank{\mat{A}}_r \rVert_{\rm F}^2.
		\end{align*}
		Last, for \NysTrace, when $r \leq m - 6$,
		\begin{align*}
			&\bigl(\mathbb{E}\, \bigl| \hat{\tr}_{\rm XN} - \tr \mat{A} \bigr|^2\bigr)^{1 \slash 2} \leq \frac{\sqrt{2}}{\sqrt{m}} \bigl(\mathbb{E} \bigl\lVert \mat{A} - \mat{A}\langle \mat{\Omega}_{-1}\rangle \bigr\rVert_{\rm F}^2\bigr)^{1 / 2} 
			+ \sqrt{2}
			\bigl(\mathbb{E} \bigl\lVert \mat{A} - \mat{A}\langle \mat{\Omega}_{-12}\rangle \bigr\rVert^2\bigr)^{1/2} \\
			&\leq \frac{\sqrt{8} m}{m - r - 5}
			\lVert \mat{A} - \lowrank{\mat{A}}_r \rVert
			\!+\! \frac{\sqrt{2} m}{(m - r - 5)^{3/2}}
			\lVert \mat{A} - \lowrank{\mat{A}}_r \rVert_{\rm F} 
			\!+\! \frac{5 \e^2 m}{(m - r - 5)^2} 
			\lVert \mat{A} - \lowrank{\mat{A}}_r \rVert_{\ast}.
		\end{align*}
		Thus, we confirm the detailed error bounds in \cref{thm:main}.
		
		All that remains is to show that each trace estimator $\hat{\tr}$ satisfies
		\begin{equation}
			\label{eq:to_verify}
			\bigl(\mathbb{E}| \hat{\tr} - \tr \mat{A} |^2\bigr)^{1/2}
			\leq \frac{C}{m} \lVert \mat{A} \rVert_\ast,
		\end{equation}
		for an absolute constant $C$. 
		To that end, apply \cref{lem:trace_relations} to bound $\lVert \mat{A} - \lowrank{\mat{A}}_r \rVert$ and $\lVert \mat{A} - \lowrank{\mat{A}}_r \rVert_{\rm F}$ in terms of $\lVert \mat{A} - \lowrank{\mat{A}}_r \rVert_{\ast}$.
		For \HutchPP, we set $r = \lfloor m / 6 \rfloor - 1$.  For \textsc{XTrace}\xspace, we set $r = \lfloor m / 4 \rfloor - 2$. For \NysTrace, we set $r = \lfloor m / 2 \rfloor - 3$.  Simplifying yields \cref{eq:to_verify} for each estimator. 
	\end{proof}
	
	\subsection{\texorpdfstring{Proof of \cref{prop:estim_bound}}{Proof of Proposition 3.1}}
	\label{sec:estim_bound}
	
	Last, we must argue that the posterior error estimator $\hat{\err}$ defined in~\cref{eq:error_est}
	reflects the actual error.  We instate the notation from \cref{prop:estim_bound}.
	
	For both \textsc{XTrace}\xspace and \NysTrace, each individual trace estimate $\hat{\tr}_i$ is unbiased.
	As a consequence, the variance takes the form
	\begin{equation*}
		\expect \mleft|  \hat{\tr} - \tr\mat{A} \mright|^2 = \Var(\hat{\tr}) = \frac{1}{\ell^2} \sum_{i,j=1}^\ell \Cov\mleft(\hat{\tr}_i,\hat{\tr}_j\mright).
	\end{equation*}
	A short calculation yields
	\begin{equation*}
		\expect \hat{\err}^2 = \frac{1}{\ell^2(\ell-1)} \sum_{i,j=1}^\ell\mleft[ \Var\mleft(\hat{\tr}_j\mright) - \Cov\mleft(\hat{\tr}_i,\hat{\tr}_j\mright)\mright].
	\end{equation*}
	Since the samples $\vec{\omega}_1,\ldots,\vec{\omega}_\ell$ are exchangeable, the variance is the same for each $j$ and the covariance is the same for all $i \ne j$.  Therefore,
	\begin{align*}
		\expect \mleft| \hat{\tr} - \tr(\mat{A}) \mright|^2 &= \frac{1}{\ell} \Var\mleft(\hat{\tr}_1\mright) + \frac{\ell-1}{\ell} \Cov\mleft(\hat{\tr}_1,\hat{\tr}_2\mright), \\
		\expect \hat{\err}^2 &=  \frac{1}{\ell} \Var\mleft(\hat{\tr}_1\mright) - \frac{1}{\ell} \Cov\mleft(\hat{\tr}_1,\hat{\tr}_2\mright).
	\end{align*}
	The result follows when we take the ratio of these two quantities and simplify. \hfill $\proofbox$
	
	As a final comment, we observe that the calculations in \cref{sec:general_error} show for \emph{symmetric matrices} that the \textsc{XTrace}\xspace correlations are bounded by
	\begin{equation*}
		\operatorname{Cor}(\hat{\tr}_1, \hat{\tr}_2) \leq 2 \frac{\expect \big\|(\Id - \mat{Q}_{(12)}^{\vphantom{*}}\mat{Q}_{(12)}^*)\mat{A}(\Id - \mat{Q}_{(12)}^{\vphantom{*}}\mat{Q}_{(12)}^*)\big\|^2}{\expect \big\|(\Id - \mat{Q}_{(1)}^{\vphantom{*}}\mat{Q}_{(1)}^*)\mat{A}(\Id - \mat{Q}_{(1)}^{\vphantom{*}}\mat{Q}_{(1)}^*)\big\|_{\rm F}^2},
	\end{equation*}
	where $\mat{Q}_{(1)}$ and $\mat{Q}_{(12)}$ are defined in \cref{prop:x}.
	These correlations are small for matrices with slow rates of singular value decay, i.e., when $\|\mat{A} - \lowrank{\mat{A}}_{m/2 - 1}\|_{\rm F} \gg \|\mat{A} - \lowrank{\mat{A}}_{m/2 - 2}\|$.
	In practice, we observe the correlations to be small even for matrices with singular values which decay more quickly.
	As an example, for the matrix with exponentially decaying eigenvalues in \cref{fig:exp_standout}, the \textsc{XTrace}\xspace correlations (measured over $10^4$ independent runs of the algorithm) are no higher than $0.06$ and the average error estimate is correct up to a factor of $1.2$.
	
	\section*{Acknowledgments}
	We thank Eitan Levin for helpful discussions regarding the fast implementation of \textsc{XTrace}\xspace.
	
	\section*{Disclaimer}
	This report was prepared as an account of work sponsored by an agency of the United States Government. Neither the United States Government nor any agency thereof, nor any of their employees, makes any warranty, express or implied, or assumes any legal liability or responsibility for the accuracy, completeness, or usefulness of any information, apparatus, product, or process disclosed, or represents that its use would not infringe privately owned rights. Reference herein to any specific commercial product, process, or service by trade name, trademark, manufacturer, or otherwise does not necessarily constitute or imply its endorsement, recommendation, or favoring by the United States Government or any agency thereof. The views and opinions of authors expressed herein do not necessarily state or reflect those of the United States Government or any agency thereof.
	
	\bibliographystyle{siamplain}
	\bibliography{refs}

\begin{thebibliography}{10}

\bibitem{AD18}
{\sc M.~Al~Hasan and V.~S. Dave}, {\em Triangle counting in large networks: A
  review}, WIREs Data Mining and Knowledge Discovery, 8 (2018), p.~e1226,
  \url{https://doi.org/10.1002/widm.1226}.

\bibitem{AH11}
{\sc A.~H. {Al-Mohy} and N.~J. Higham}, {\em Computing the action of the matrix
  exponential, with an application to exponential integrators}, SIAM Journal on
  Scientific Computing, 33 (2011), pp.~488--511,
  \url{https://doi.org/10.1137/100788860}.

\bibitem{AMT10}
{\sc H.~Avron, P.~Maymounkov, and S.~Toledo}, {\em Blendenpik: {{Supercharging
  LAPACK}}'s least-squares solver}, SIAM Journal on Scientific Computing, 32
  (2010), pp.~1217--1236, \url{https://doi.org/10.1137/090767911}.

\bibitem{BN22}
{\sc R.~A. Baston and Y.~Nakatsukasa}, {\em Stochastic diagonal estimation:
  Probabilistic bounds and an improved algorithm}, Jan. 2022,
  \url{https://arxiv.org/abs/2201.10684v1}.

\bibitem{BKS07}
{\sc C.~Bekas, E.~Kokiopoulou, and Y.~Saad}, {\em An estimator for the diagonal
  of a matrix}, Applied Numerical Mathematics, 57 (2007), pp.~1214--1229,
  \url{https://doi.org/10.1016/j.apnum.2007.01.003}.

\bibitem{BZC+03}
{\sc D.~Bu, Y.~Zhao, L.~Cai, H.~Xue, X.~Zhu, H.~Lu, J.~Zhang, S.~Sun, L.~Ling,
  N.~Zhang, G.~Li, and R.~Chen}, {\em Topological structure analysis of the
  protein\textendash protein interaction network in budding yeast}, Nucleic
  Acids Research, 31 (2003), pp.~2443--2450,
  \url{https://doi.org/10.1093/nar/gkg340}.

\bibitem{CH22}
{\sc T.~Chen and E.~Hallman}, {\em Krylov-aware stochastic trace estimation},
  Nov. 2022, \url{https://arxiv.org/abs/2205.01736v2}.

\bibitem{CTU22}
{\sc T.~Chen, T.~Trogdon, and S.~Ubaru}, {\em Randomized matrix-free quadrature
  for spectrum and spectral sum approximation}, Sept. 2022,
  \url{https://arxiv.org/abs/2204.01941v2}.

\bibitem{DH11}
{\sc T.~Davis and Y.~Hu}, {\em The {{University}} of {{Florida}} sparse matrix
  collection}, ACM Transactions on Mathematical Software, 38 (2011), pp.~1--25,
  \url{https://doi.org/10.1145/2049662.2049663}.

\bibitem{ET22}
{\sc E.~N. Epperly and J.~A. Tropp}, {\em Efficient error and variance
  estimation for randomized matrix computations}, Mar. 2023,
  \url{https://arxiv.org/abs/2207.06342v3}.

\bibitem{Est22}
{\sc E.~Estrada}, {\em The many facets of the {{Estrada}} indices of graphs and
  networks}, SeMA Journal, 79 (2022), pp.~57--125,
  \url{https://doi.org/10.1007/s40324-021-00275-w}.

\bibitem{GSO17}
{\sc A.~S. Gambhir, A.~Stathopoulos, and K.~Orginos}, {\em Deflation as a
  method of variance reduction for estimating the trace of a matrix inverse},
  SIAM Journal on Scientific Computing, 39 (2017), pp.~A532--A558,
  \url{https://doi.org/10.1137/16M1066361}.

\bibitem{GSTV07}
{\sc A.~C. Gilbert, M.~J. Strauss, J.~A. Tropp, and R.~Vershynin}, {\em One
  sketch for all: Fast algorithms for compressed sensing}, in Proceedings of
  the Thirty-Ninth Annual {{ACM}} Symposium on {{Theory}} of Computing, June
  2007, pp.~237--246, \url{https://doi.org/10.1145/1250790.1250824}.

\bibitem{Gir89}
{\sc A.~Girard}, {\em A fast``{{Monte-Carlo}} cross-validation'' procedure for
  large least squares problems with noisy data}, Numerische Mathematik, 56
  (1989), pp.~1--23, \url{https://doi.org/10.1007/BF01395775}.

\bibitem{HMT11}
{\sc N.~Halko, P.-G. Martinsson, and J.~A. Tropp}, {\em Finding structure with
  randomness: Probabilistic algorithms for constructing approximate matrix
  decompositions}, SIAM Review, 53 (2011), pp.~217--288,
  \url{https://doi.org/10.1137/090771806}.

\bibitem{Hal46}
{\sc P.~R. Halmos}, {\em The theory of unbiased estimation}, The Annals of
  Mathematical Statistics, 17 (1946), pp.~34--43,
  \url{https://doi.org/10.1214/aoms/1177731020}.

\bibitem{Hig08}
{\sc N.~J. Higham}, {\em Functions of Matrices: Theory and Computation},
  {SIAM}, {Philadelphia}, 2008, \url{https://doi.org/10.1137/1.9780898717778}.

\bibitem{Hig10}
{\sc N.~J. Higham}, {\em Matrix exponential times a vector}, Nov. 2010.
\newblock {Available} at
  \url{https://www.mathworks.com/matlabcentral/fileexchange/29576-matrix-exponential-times-a-vector}
  (accessed 10/25/2022).

\bibitem{Hut89}
{\sc M.~F. Hutchinson}, {\em A stochastic estimator of the trace of the
  influence matrix for {Laplacian} smoothing splines}, Communications in
  Statistics - Simulation and Computation, 18 (1989), pp.~1059--1076,
  \url{https://doi.org/10.1080/03610918908812806}.

\bibitem{JPWZ21a}
{\sc S.~Jiang, H.~Pham, D.~P. Woodruff, and Q.~Zhang}, {\em Optimal sketching
  for trace estimation}, in 35th {{Conference}} on {{Neural Information
  Processing Systems}}., 2021, p.~13.

\bibitem{KB94}
{\sc V.~S. Koroljuk and Y.~V. Borovskich}, {\em Theory of {{U-Statistics}}},
  {Springer Netherlands}, {Dordrecht}, 1994,
  \url{https://doi.org/10.1007/978-94-017-3515-5}.

\bibitem{LLS+17}
{\sc H.~Li, G.~C. Linderman, A.~Szlam, K.~P. Stanton, Y.~Kluger, and
  M.~Tygert}, {\em Algorithm 971: An implementation of a randomized algorithm
  for principal component analysis}, ACM Transactions On Mathematical Software,
  43 (2017), \url{https://doi.org/10.1145/3004053}.

\bibitem{Lin17}
{\sc L.~Lin}, {\em Randomized estimation of spectral densities of large
  matrices made accurate}, Numerische Mathematik, 136 (2017), pp.~183--213,
  \url{https://doi.org/10.1007/s00211-016-0837-7}.

\bibitem{Lit19}
{\sc C.~Litens}, {\em Transverse {{Field Ising Model}} with Different Boundary
  Conditions}, {Bachelor's Thesis}, Stockholm University, Feb. 2019.
\newblock Available at
  \url{http://staff.fysik.su.se/~ardonne/files/theses/bachelor-thesis_christopher-litens.pdf}.

\bibitem{MT20a}
{\sc P.-G. Martinsson and J.~A. Tropp}, {\em Randomized numerical linear
  algebra: {{Foundations}} and algorithms}, Acta Numerica, 29 (2020),
  pp.~403--572, \url{https://doi.org/10.1017/S0962492920000021}.

\bibitem{MMMW21}
{\sc R.~A. Meyer, C.~Musco, C.~Musco, and D.~P. Woodruff}, {\em Hutch++:
  Optimal stochastic trace estimation}, in Symposium on Simplicity in
  Algorithms, SIAM, Jan. 2021.

\bibitem{MMMW21a}
{\sc R.~A. Meyer, C.~Musco, C.~Musco, and D.~P. Woodruff}, {\em Hutch++:
  {{Optimal}} stochastic trace estimation}, June 2021,
  \url{https://arxiv.org/abs/2002.11457v5}.

\bibitem{PCK22}
{\sc D.~Persson, A.~Cortinovis, and D.~Kressner}, {\em Improved variants of the
  {Hutch++} algorithm for trace estimation}, SIAM Journal on Matrix Analysis
  and Applications,  (2022), pp.~1162--1185,
  \url{https://doi.org/10.1137/21M1447623}.

\bibitem{Pfe70}
{\sc P.~Pfeuty}, {\em The one-dimensional {{Ising}} model with a transverse
  field}, Annals of Physics, 57 (1970), pp.~79--90,
  \url{https://doi.org/10.1016/0003-4916(70)90270-8}.

\bibitem{RT08}
{\sc V.~Rokhlin and M.~Tygert}, {\em A fast randomized algorithm for
  overdetermined linear least-squares regression}, Proceedings of the National
  Academy of Sciences, 105 (2008), pp.~13212--13217,
  \url{https://doi.org/10.1073/pnas.0804869105}.

\bibitem{SAI17}
{\sc A.~K. Saibaba, A.~Alexanderian, and I.~C.~F. Ipsen}, {\em Randomized
  matrix-free trace and log-determinant estimators}, Numerische Mathematik, 137
  (2017), pp.~353--395, \url{https://doi.org/10.1007/s00211-017-0880-z}.

\bibitem{TW23a}
{\sc J.~A. Tropp and R.~J. Webber}, {\em Randomized algorithms for low-rank
  matrix approximation: {Design}, analysis, and applications}, June 2023,
  \url{https://doi.org/10.48550/arXiv.2306.12418v1}.

\bibitem{TYUC17b}
{\sc J.~A. Tropp, A.~Yurtsever, M.~Udell, and V.~Cevher}, {\em Fixed-rank
  approximation of a positive-semidefinite matrix from streaming data}, in
  Advances in Neural Information Processing Systems, vol.~30, 2017,
  pp.~1225--1234.

\bibitem{UCS17}
{\sc S.~Ubaru, J.~Chen, and Y.~Saad}, {\em Fast estimation of
  $\operatorname{tr}(f({{A}}))$ via stochastic {Lanczos} quadrature}, SIAM
  Journal on Matrix Analysis and Applications, 38 (2017), pp.~1075--1099,
  \url{https://doi.org/10.1137/16M1104974}.

\bibitem{US17}
{\sc S.~Ubaru and Y.~Saad}, {\em Applications of trace estimation techniques},
  in High {{Performance Computing}} in {{Science}} and {{Engineering}}, 2017,
  \url{https://doi.org/10.1007/978-3-319-97136-0_2}.

\end{thebibliography}
	
	\newpage
	\begin{center} \textbf{SUPPLEMENTARY MATERIAL}
	\end{center}
	
	\setcounter{section}{0} \renewcommand{\thesection}{SM\arabic{section}}
	\setcounter{figure}{0} \renewcommand{\thefigure}{SM\arabic{figure}}
	\renewcommand\appendixname{Section}
	
	\section{Comparison with Adaptive \HutchPP} \label{sec:compare_adapt}
	
	The adaptive \HutchPP algorithm of \cite{PCK22} flexibly apportions test vectors between low-rank approximation and residual trace estimation, based on the estimated singular values of the input matrix.
	This algorithm interpolates between the Girard--Hutchinson estimator, the \HutchPP estimator, and a purely low-rank approximation-based trace estimate.
	This algorithm also estimates the total number of matvecs $m$ to meet an (absolute) error tolerance:
	\begin{equation*}
		\mleft| \hat{\tr} - \tr(\mat{A}) \mright| \le \varepsilon_{\rm abs} \quad \text{except with probability $1-2\delta$},
	\end{equation*}
	where $\delta$ is a parameter chosen by the user.
	
	We compare \textsc{XTrace}\xspace with adaptive \HutchPP in \cref{fig:adapt}.
	To produce these plots, we first run adaptive \HutchPP with error tolerances $\varepsilon_{\rm abs} = \varepsilon \cdot \tr \mat{A}$ and failure probability parameter $\delta \coloneqq 1/10$.
	For each value of the error tolerance $\epsilon$, we plot the mean error (solid line) and desired accuracy $\varepsilon$ against the mean number $m$ of matvecs required by the algorithm.
	We then run \textsc{XTrace}\xspace with a variable number of matvecs $m$ and report the mean error over $1000$ trials.
	Overall, we find that \textsc{XTrace}\xspace performs similarly to adaptive \HutchPP or better than adaptive \HutchPP by up to an order of magnitude.
	
	\begin{figure}[t]
		\centering
		\begin{subfigure}[b]{0.48\textwidth}
			\centering
			\includegraphics[width=\textwidth]{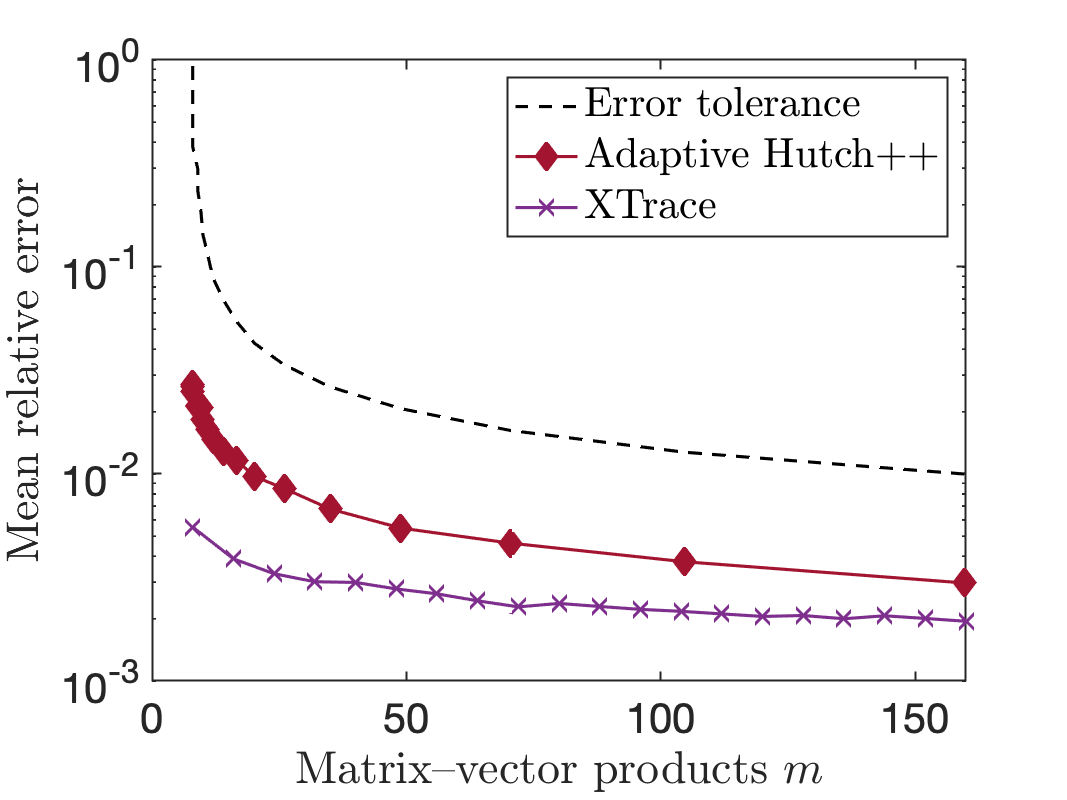}
			\caption{\texttt{flat}} \label{fig:flat_adapt}
		\end{subfigure}
		~
		\begin{subfigure}[b]{0.48\textwidth}
			\centering
			\includegraphics[width=\textwidth]{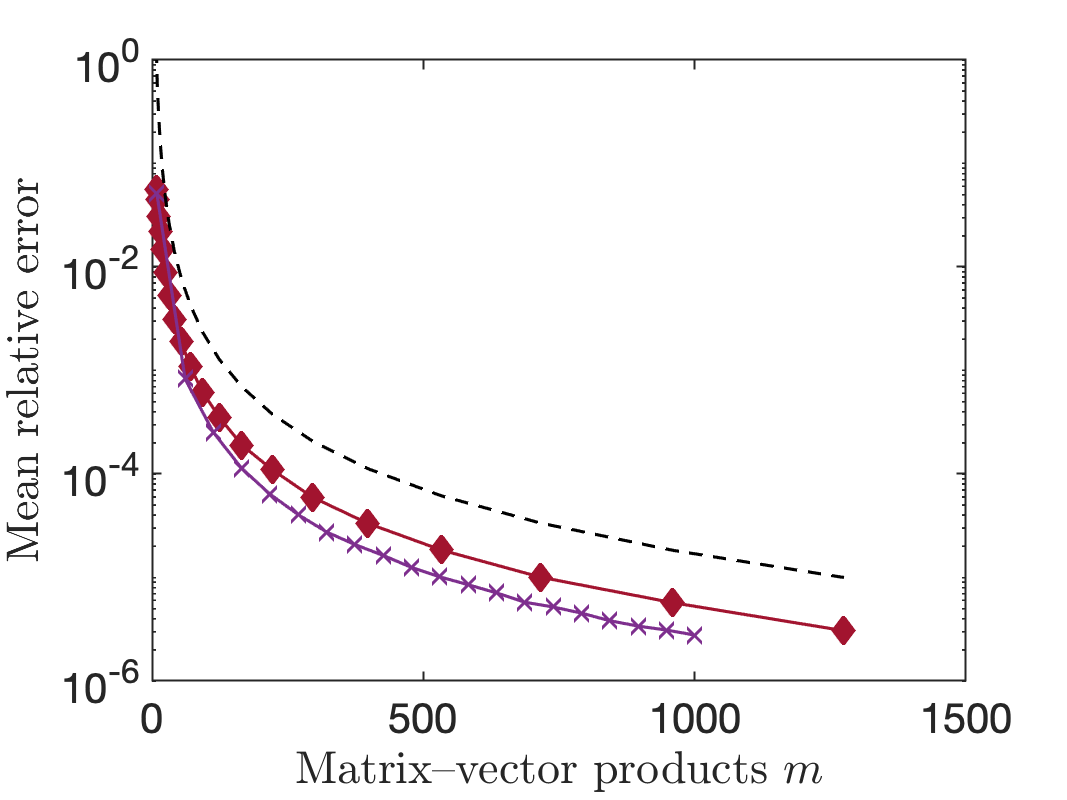}
			\caption{\texttt{poly}} \label{fig:poly_adapt}
		\end{subfigure}
		
		\begin{subfigure}[b]{0.48\textwidth}
			\centering
			\includegraphics[width=\textwidth]{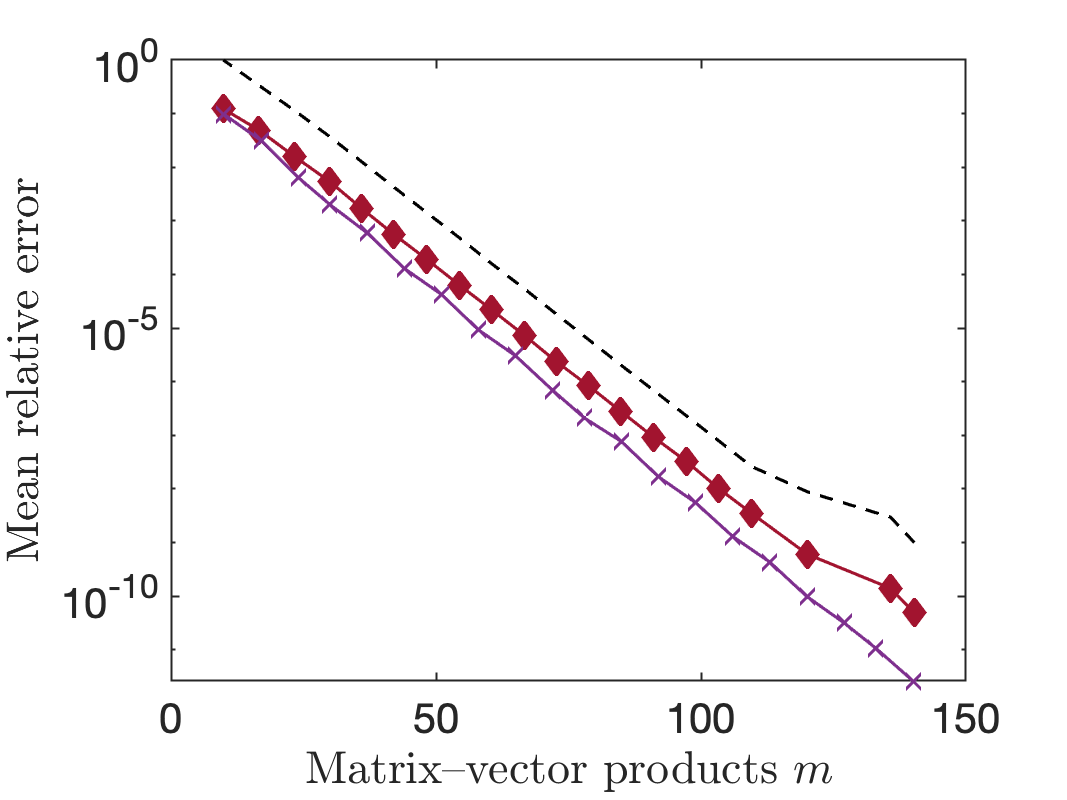}
			\caption{\texttt{exp}} \label{fig:exp_adapt}
		\end{subfigure}
		\begin{subfigure}[b]{0.48\textwidth}
			\centering
			\includegraphics[width=\textwidth]{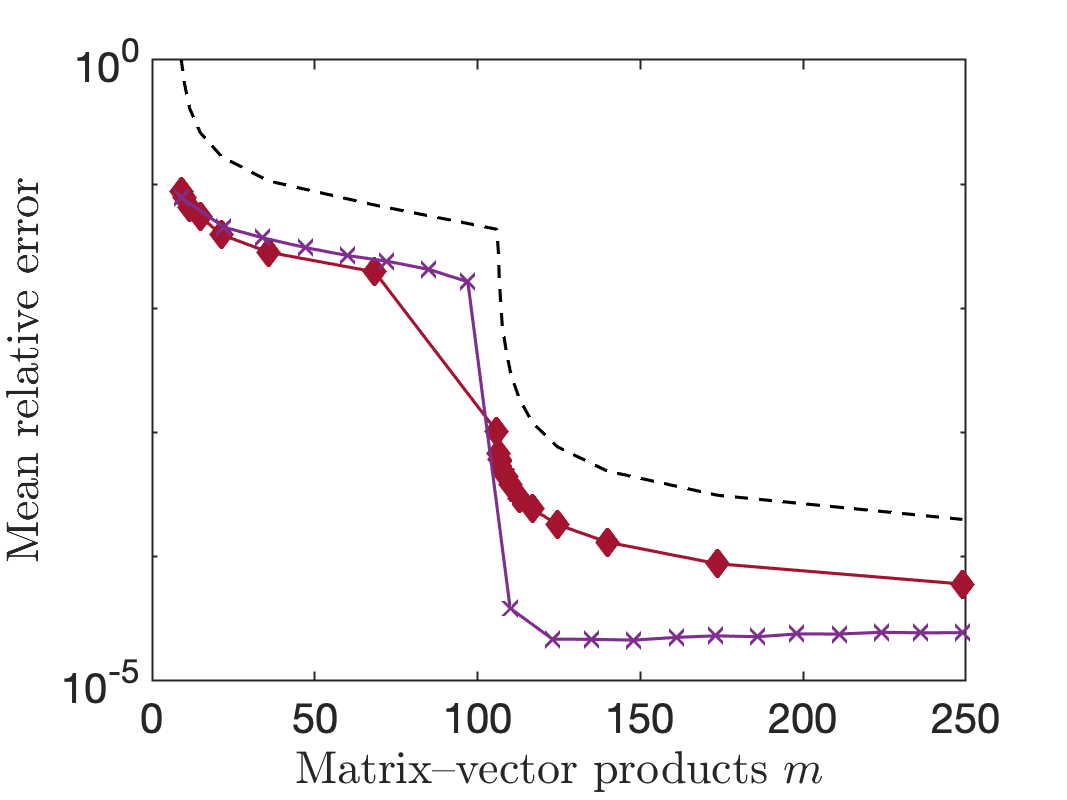}
			\caption{\texttt{step}} \label{fig:step_adapt}
		\end{subfigure}
		\caption{\textbf{Comparison with adaptive \HutchPP.}  Average relative error using \textsc{XTrace}\xspace and adaptive \HutchPP (solid lines) and user-chosen error tolerance for Adaptive \HutchPP (dashed lines). See~\cref{sec:compare_adapt}.} \label{fig:adapt}
	\end{figure}
	
	\section{Runtime comparison} \label{sec:time}
	
	Throughout the main body of the text, we use the number of matrix--vector products as a surrogate for the runtime of different trace estimators, which is justified by the observation that matvecs frequently dominate the cost of trace estimation.
	For instance, for the quantum statistical physics examples in \cref{sec:intro-synth-exper,sec:quantum}, matvecs are responsible for over 90\% of the runtime for all of the trace estimators.
	However, when matvec operations are fairly cheap or the number of matvecs is large, the cost of \emph{processing} the matvecs can come to dominate the runtime.
	
	To illustrate the runtime differences between the algorithms, let us compare the Girard--Hutchinson estimator, \HutchPP, and \textsc{XTrace}\xspace.
	The only processing cost of the Girard--Hutchinson algorithm comes from the computation of the inner products $\vec{\omega}_i^*(\mat{A}\vec{\omega}_i)$ between the matvecs $\mat{A}\vec{\omega}_i$ and the test vectors $\vec{\omega}_i$ for $i = 1,\ldots,m$, requiring just $\order(mN)$ operations.
	By contrast, \HutchPP orthogonalizes a matrix of size $N \times (m/3)$, resulting in a higher processing cost of $\order(m^2N)$ operations.
	For a fixed budget of $m$ matvecs, \textsc{XTrace}\xspace has the highest processing cost, since the algorithm orthogonalizes a larger $N\times (m/2)$ matrix.
	Yet, \textsc{XTrace}\xspace often compensates for its high processing cost because it results in lower error than the other algorithms.
	
	\begin{figure}[ht!]
		\centering
		\texttt{flat}
		
		\begin{subfigure}[b]{0.39\textwidth}
			\centering
			\includegraphics[width=\textwidth]{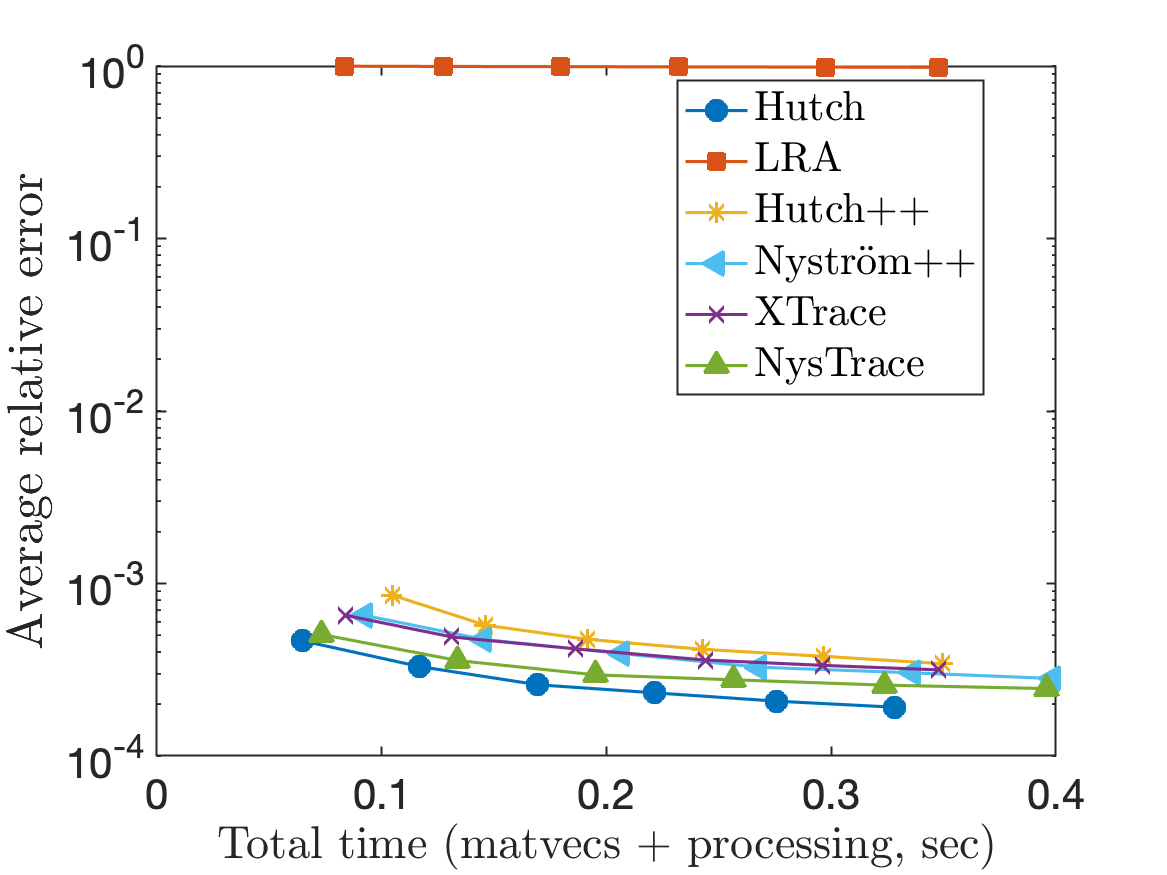} \label{fig:flat_time}
		\end{subfigure}
		~
		\begin{subfigure}[b]{0.39\textwidth}
			\centering
			\includegraphics[width=\textwidth]{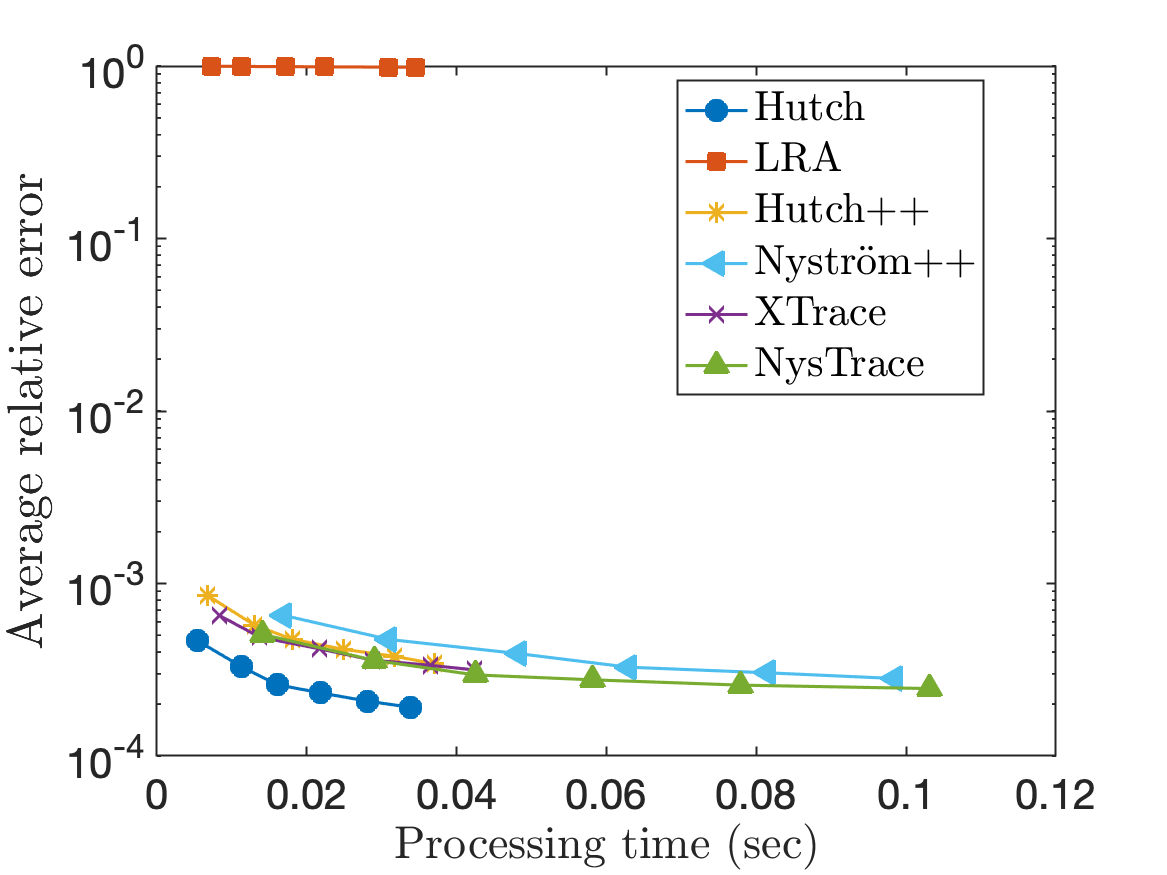} \label{fig:flat_prep}
		\end{subfigure}
		
		\texttt{poly}
		
		\begin{subfigure}[b]{0.39\textwidth}
			\centering
			\includegraphics[width=\textwidth]{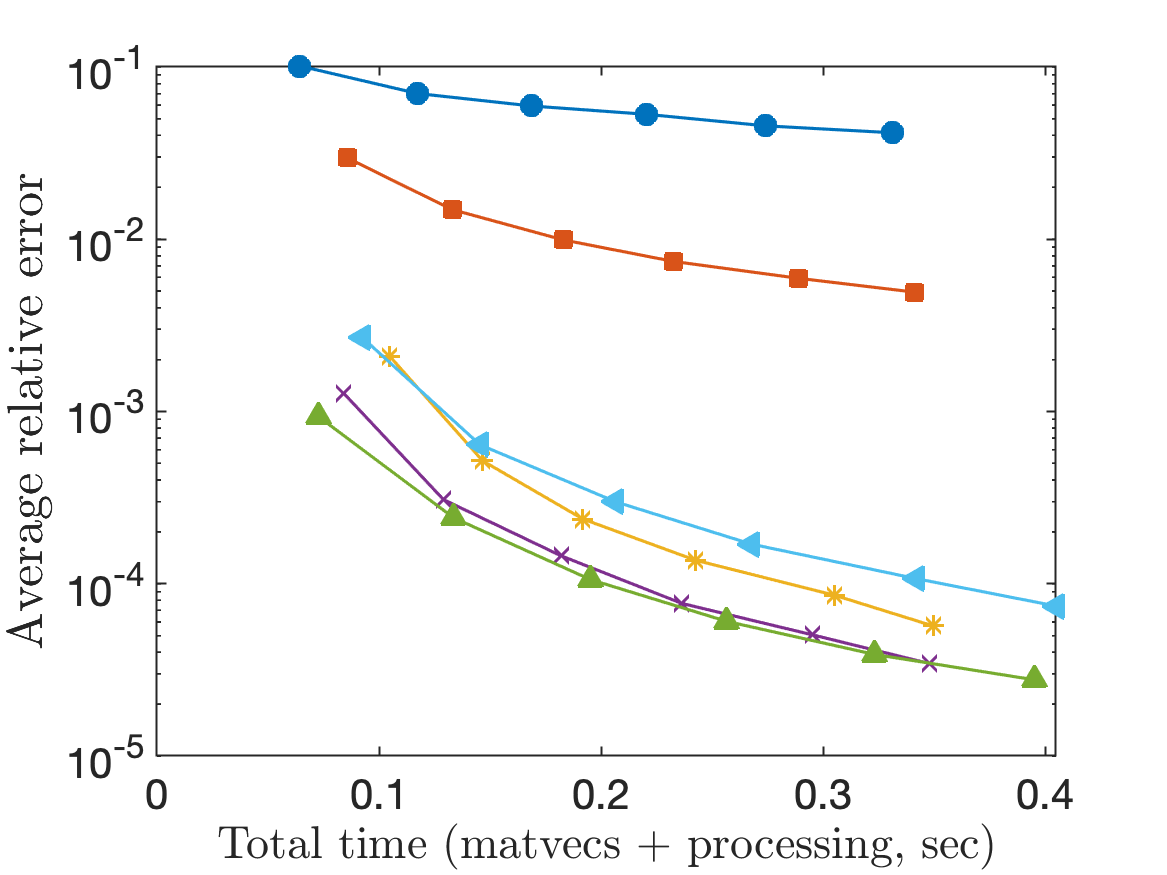}
			\label{fig:poly_time}
		\end{subfigure}
		\begin{subfigure}[b]{0.39\textwidth}
			\centering
			\includegraphics[width=\textwidth]{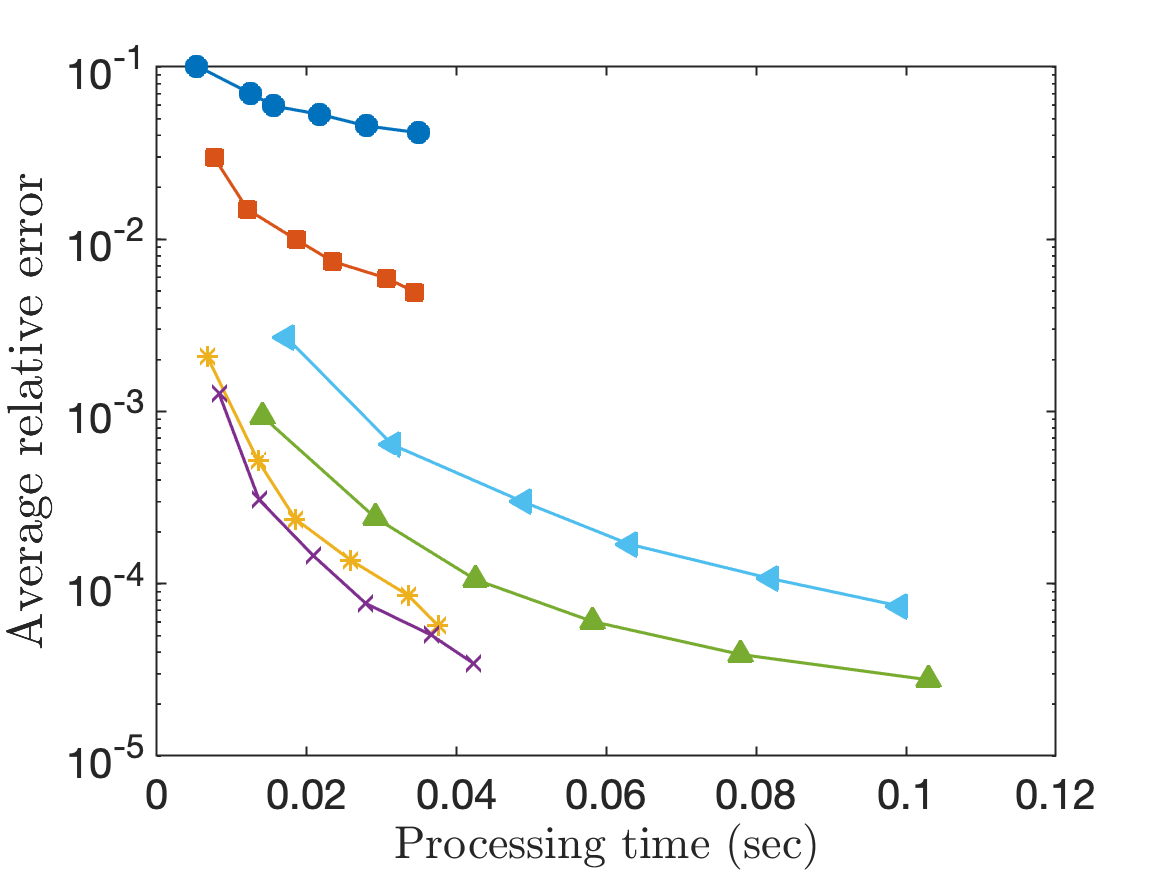} \label{fig:poly_prep}
		\end{subfigure}
		
		\texttt{exp}
		
		\begin{subfigure}[b]{0.39\textwidth}
			\centering
			\includegraphics[width=\textwidth]{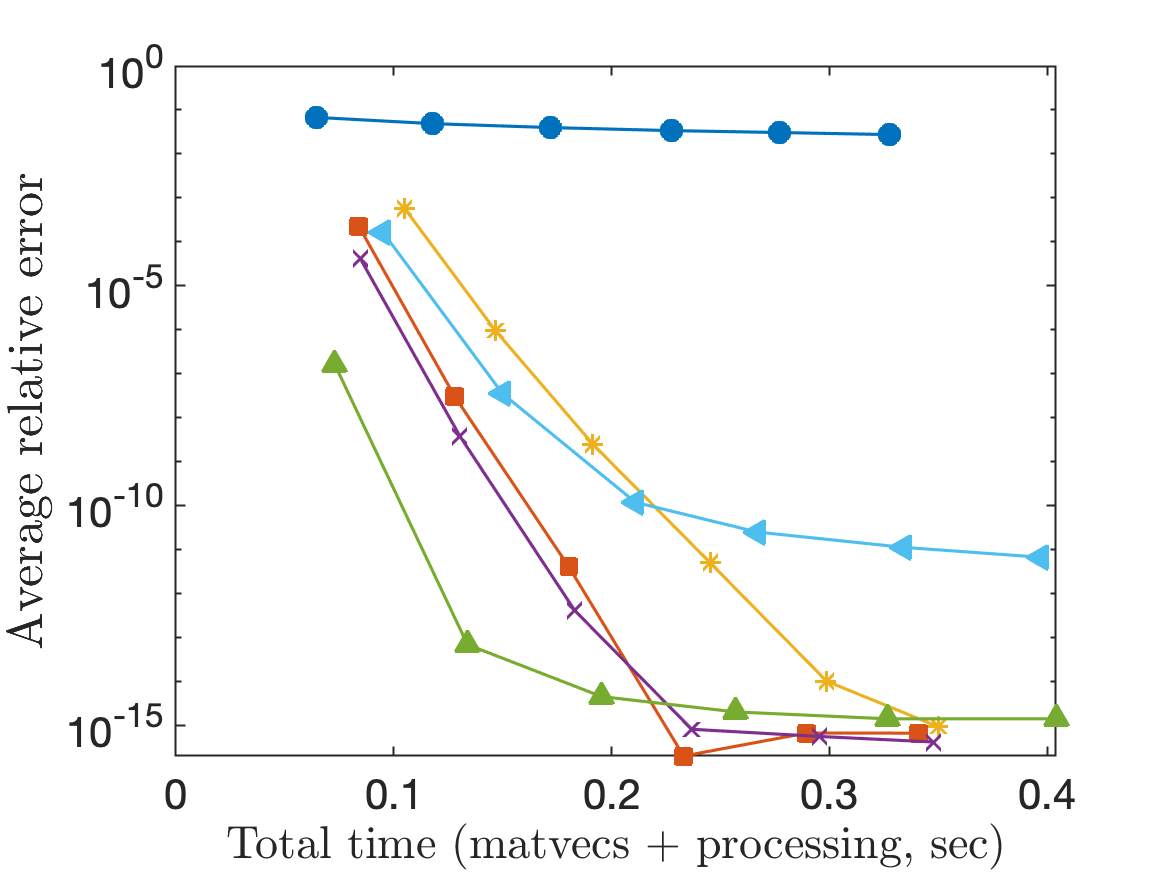}
			\label{fig:fastexp_time}
		\end{subfigure}
		\begin{subfigure}[b]{0.39\textwidth}
			\centering
			\includegraphics[width=\textwidth]{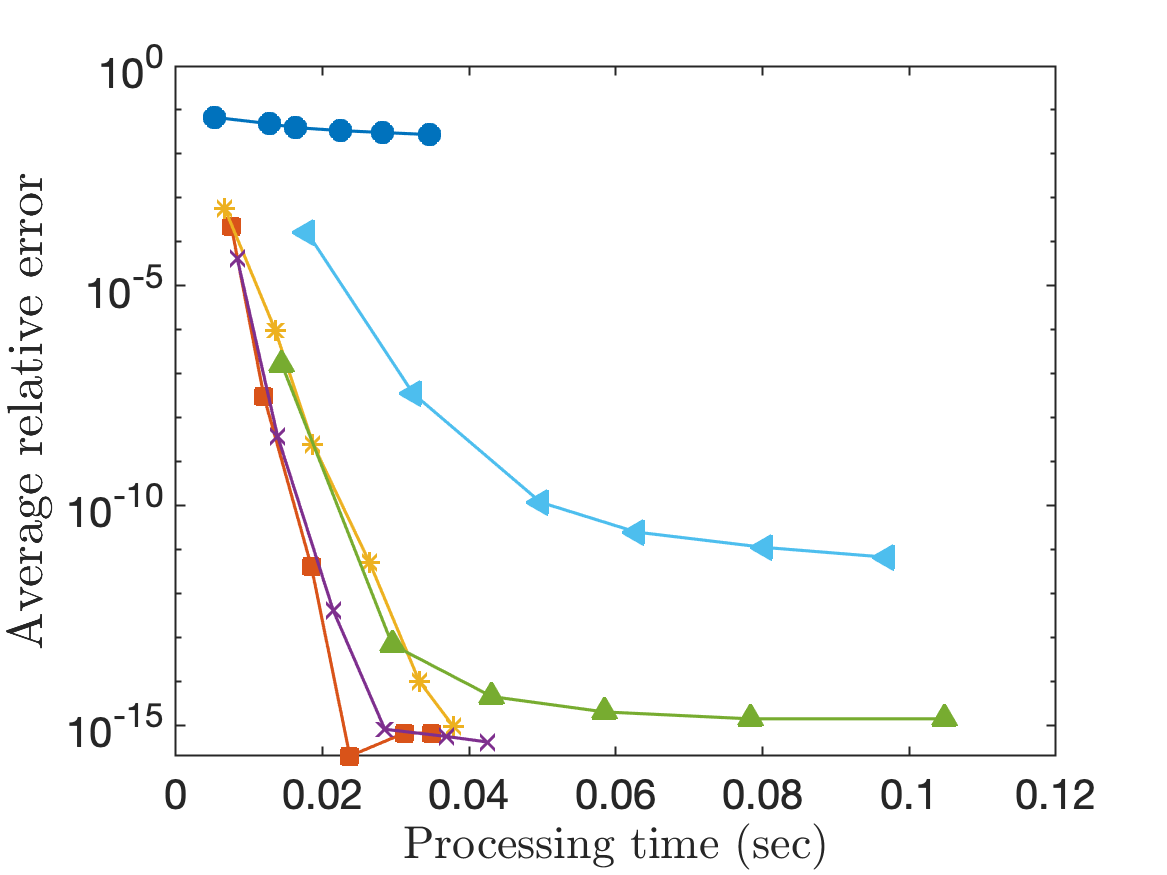} \label{fig:fastexp_prep}
		\end{subfigure}
		
		\texttt{step}
		
		\begin{subfigure}[b]{0.39\textwidth}
			\centering
			\includegraphics[width=\textwidth]{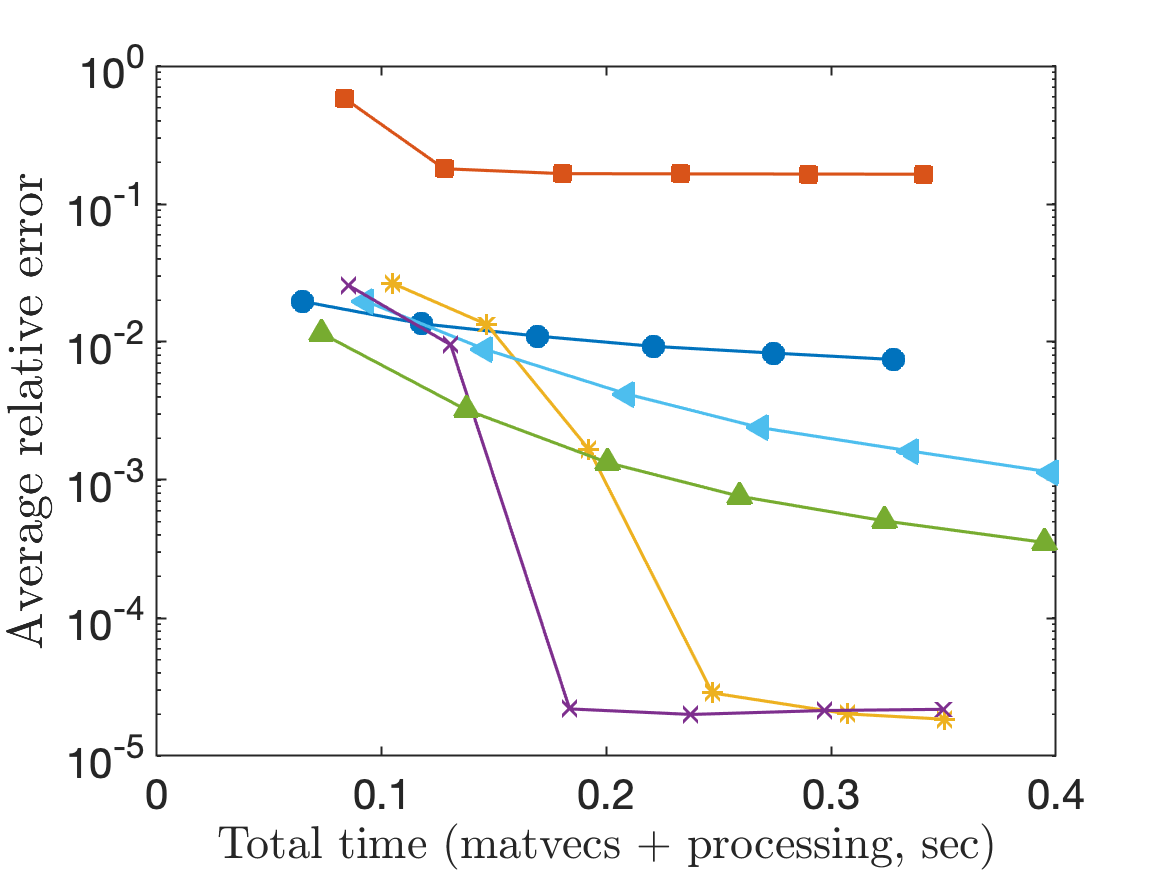}
			\label{fig:smallstep_time}
		\end{subfigure}
		\begin{subfigure}[b]{0.39\textwidth}
			\centering
			\includegraphics[width=\textwidth]{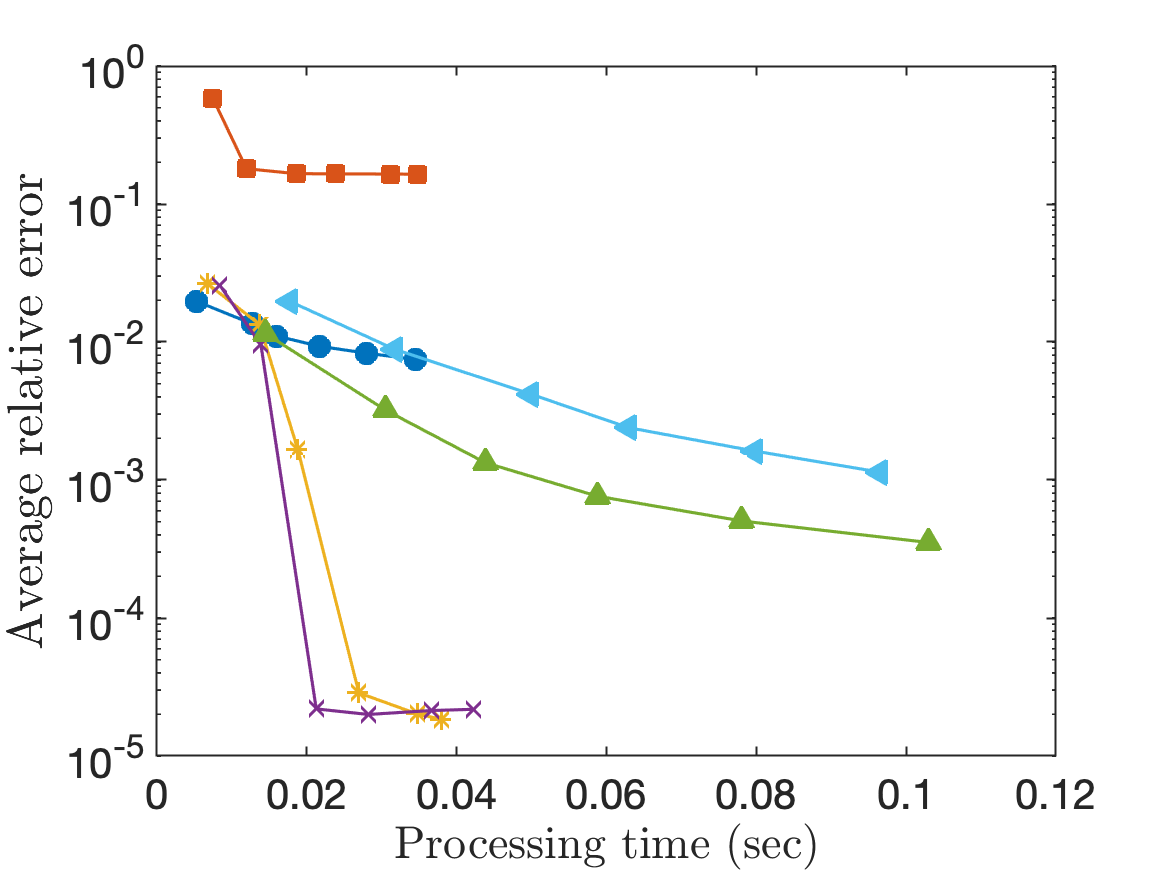} \label{fig:smallstep_prep}
		\end{subfigure}
		\caption{\textbf{Runtime.} Average relative error versus runtime (\emph{left}) and processing time (runtime minus matvecs, \emph{right}); see \cref{sec:time}.} \label{fig:time}
	\end{figure}
	
	\Cref{fig:time} shows the error versus the runtime for different stochastic trace estimators applied to matrices with several spectral decay patterns.
	We use the same spectral profiles as in \cref{sec:comparisons}, but use a larger size $N = 10^4$.
	In the left panels, we show the error as a function of the total runtime, whereas the right panels show the processing time (total runtime minus the time required to perform the matvecs).
	Runtimes and errors are computed by averaging over 1000 independent trials. 
	
	In these experiments, \textsc{XTrace}\xspace consistently achieves lower errors than \HutchPP for a given amount of runtime or processing time, despite the additional costs associated with orthogonalizing a larger matrix.
	This shows that \textsc{XTrace}\xspace can be a more effective trace estimator than \HutchPP even if processing costs begin to dominate the runtime.
	The processing cost of the Nystr\"om-based trace estimators \NysTrace and \NysPP is significantly higher than \HutchPP and \textsc{XTrace}\xspace.
	This suggests that when matvecs are expensive, \NysTrace should be preferred over \textsc{XTrace}\xspace and when matvecs are cheap, it may be worth using \textsc{XTrace}\xspace to avoid \NysTrace's higher processing costs.
	
	\section{Derivation for efficient \textsc{XTrace}\xspace implementation}
	\label{sec:xtrace_derivation}
	
	In the main text, we proved the update formula
	\begin{equation} \label{eq:update_formula}
		\mat{Q}_{(i)}^{\vphantom{*}}\mat{Q}_{(i)}^* = \mat{Q} (\Id - \vec{s}_{i}^{\vphantom{*}}\vec{s}_i^*)\mat{Q}^*
	\end{equation}
	for the the $\mat{Q}_{(i)}$ matrices appearing in the \textsc{XTrace}\xspace estimator and devised an $\order(m^3)$ procedure for computing $\mat{S} = \begin{bmatrix} \vec{s}_1 & \cdots \vec{s}_{m/2}\end{bmatrix}$. 
	In this section, we complete the derivation of how to compute the basic trace estimators $\hat{\tr}_i$ in $\order(m^2N)$ operations.
	
	Instate the notation of \cref{sec:xtrace}.
	The basic trace estimator $\hat{\tr}_i$ is defined as
	\begin{equation*}
		\hat{\tr}_i = \tr \mleft( \mat{Q}_{(i)}^* \mleft( \mat{A}\mat{Q}^{\vphantom{*}}_{(i)}\mright) \mright) + \vec{\omega}_i^* \mleft( \Id - \mat{Q}_{(i)}^{\vphantom{*}}\mat{Q}_{(i)}^*\mright) \mat{A} \mleft( \Id - \mat{Q}_{(i)}^{\vphantom{*}}\mat{Q}_{(i)}^*\mright)\vec{\omega}_i^{\vphantom{*}}
	\end{equation*}
	Now invoke \cref{eq:update_formula}:
	\begin{equation*}
		\hat{\tr}_i = \tr \mleft( \mat{A}\mat{Q} (\Id - \vec{s}_{i}^{\vphantom{*}}\vec{s}_i^*)\mat{Q}^*  \mright) + \vec{\omega}_i^* \mleft( \Id - \mat{Q} (\Id - \vec{s}_{i}^{\vphantom{*}}\vec{s}_i^*)\mat{Q}^*\mright) \mat{A} \mleft( \Id - \mat{Q} (\Id - \vec{s}_{i}^{\vphantom{*}}\vec{s}_i^*)\mat{Q}^*\mright)\vec{\omega}_i^{\vphantom{*}}.
	\end{equation*}
	Define
	\begin{equation} \label{eq:auxilliary_matrices}
		\mat{Z} \coloneqq \mat{A}\mat{Q}, \quad \mat{H} \coloneqq \mat{Q}^*\mat{Z}, \quad \mat{W} \coloneqq \mat{Q}^* \mat{\Omega}, \quad \mat{T} \coloneqq \mat{Z}^*\mat{\Omega},
	\end{equation}
	all of which can be computed using $m/2$ matvecs and $\order(m^2N)$ operations.
	Denote the columns of a matrix using the corresponding lowercase letter, e.g., the $i$th column of $\mat{W}$ is $\vec{w}_i$.
	Using the newly introduced variables, the trace estimates become
	\begin{align*}
		\hat{\tr}_i &= \tr (\mat{H}) - \vec{s}_i^*\mat{H}\vec{s}_i^{\vphantom{*}} + \vec{\omega}_i^* (\Id - \mat{Q}\mat{Q}^*)\mat{A}\vec{\omega}_i + \vec{w}_i^*\vec{s}_i^{\vphantom{*}} \vec{s}_i^*\mat{Q}^*\mat{A}\vec{\omega}_i^{\vphantom{*}} \\ 
		&\qquad- \mleft( \vec{\omega}_i^* - \vec{w}_i^* (\Id - \vec{s}_{i}^{\vphantom{*}}\vec{s}_i^*)\mat{Q}^*\mright) \mat{Z} (\Id - \vec{s}_{i}^{\vphantom{*}}\vec{s}_i^*)\vec{w}_i^{\vphantom{*}}.
	\end{align*}
	To simplify this, we make two observations.
	First, $\mat{Q}\mat{Q}^*$ is an orthoprojector onto the range of $\mat{A}\mat{\Omega}$, so $(\Id - \mat{Q}\mat{Q}^*)\mat{A}\vec{\omega}_i = \vec{0}$.
	Second, recall that we defined and factorized $\mat{Y} = \mat{A}\mat{\Omega} = \mat{Q}\mat{R}$.
	Thus, $\mat{Q}^*\mat{A}\vec{\omega}_i = \mat{Q}^*\vec{y}_i = \vec{r}_i$.
	Adding these simplifications, we have
	\begin{align*}
		\hat{\tr}_i &= \tr (\mat{H}) - \vec{s}_i^*\mat{H}\vec{s}_i^{\vphantom{*}} + \vec{w}_i^*\vec{s}_i^{\vphantom{*}}\cdot \vec{s}_i^*\vec{r}_i^{\vphantom{*}} - \mleft( \vec{\omega}_i^* - \vec{w}_i^* (\Id - \vec{s}_{i}^{\vphantom{*}}\vec{s}_i^*)\mat{Q}^*\mright) \mat{Z} (\Id - \vec{s}_{i}^{\vphantom{*}}\vec{s}_i^*)\vec{w}_i^{\vphantom{*}}.
	\end{align*}
	Now, define
	\begin{equation*}
		\vec{x}_i^{\vphantom{*}} \coloneqq (\Id - \vec{s}_i^{\vphantom{*}}\vec{s}_i^*)\vec{w}_i^{\vphantom{*}} = \vec{w}_i^{\vphantom{*}} - \vec{s}_i^*\vec{w}_i^{\vphantom{*}} \cdot \vec{s}_i \quad \text{for $i=1,2,\ldots,m/2$}.
	\end{equation*}
	Using the new $\vec{x}_i$ notation and the definitions \cref{eq:auxilliary_matrices}, we conclude
	\begin{equation} \label{eq:simplified_formula}
		\hat{\tr}_i = \tr (\mat{H}) - \vec{s}_i^*\mat{H}\vec{s}_i^{\vphantom{*}} + \vec{w}_i^*\vec{s}_i^{\vphantom{*}}\cdot  \vec{s}_i^*\vec{r}_i^{\vphantom{*}} - \vec{t}_i^*\vec{x}_i^{\vphantom{*}} + \vec{x}_i^*\mat{H}\vec{x}_i^{\vphantom{*}}.
	\end{equation}
	
	The update formula \cref{eq:simplified_formula} is used to implement \textsc{XTrace}\xspace in $\order(m^2N)$ operations in \cref{alg:xtrace_efficient}.
	\Cref{alg:xtrace_efficient} is essentially identical to the MATLAB implementation provided in \cref{list:xtrace}, except that \cref{list:xtrace} uses normalized spherically symmetric test vectors (see \cref{sec:distribution}) and uses vectorized operations rather than a for loop to compute the basic estimators $\hat{\tr}_i$.
	
	\begin{algorithm}[t]
		\caption{\textsc{XTrace}\xspace: Efficient implementation \label{alg:xtrace_efficient}}
		\begin{algorithmic}[1]
			\Require Matrix $\mat{A} \in \real^{N \times N}$ and number $m$ of matvecs, where $m$ is even
			\Ensure Trace estimate $\hat{\tr} \approx \tr \mat{A}$ and error estimate $\hat{\err} \approx \vert\hat{\tr} - \tr \mat{A}\vert$
			\State Draw $\mat{\Omega}\sim \Unif\{\pm 1\}^{N\times m/2}$
			\State $\mat{Y} \leftarrow \mat{A}\mat{\Omega}$ \Comment{$m/2$ matvecs}
			\State $(\mat{Q},\mat{R}) \leftarrow \mathtt{qr}(\mat{Y},\texttt{'econ'})$ \Comment{$\order(m^2N)$ operations}
			\State $\mat{Z} \leftarrow \mat{A}\mat{Q}$ \Comment{$m/2$ matvecs}
			\State $\mat{H} \leftarrow \mat{Q}^*\mat{Z}$, $\mat{W} \leftarrow \mat{Q}^*\mat{\Omega}$, $\mat{T} \leftarrow \mat{Z}^*\mat{\Omega}$ 
			\Comment{$\order(m^2N)$ operations}
			\State $\mat{S} \leftarrow (\mat{R}^*)^{-1}$ \Comment{$\order(m^3)$ operations}
			\State $\mat{S} \leftarrow \mat{S} \cdot \diag(\norm{\vec{s}_i}^{-1} : i=1,\ldots,m/2)$ \Comment{$\order(m^2)$ operations}
			\For{$i=1,\ldots,m/2$}
			\State $\vec{x}_i \leftarrow \vec{w}_i - \vec{s}_i^*\vec{w}_i \cdot \vec{s}_i$ \Comment{$\order(m)$ operations}
			\State $\hat{\tr}_i \leftarrow \tr (\mat{H}) - \vec{s}_i^*(\mat{H}\vec{s}_i^{\vphantom{*}}) + \vec{w}_i^*\vec{s}_i^{\vphantom{*}}\cdot  \vec{s}_i^*\vec{r}_i^{\vphantom{*}} - \vec{t}_i^*\vec{x}_i^{\vphantom{*}} + \vec{x}_i^*(\mat{H}\vec{x}_i^{\vphantom{*}})$ \Comment{$\order(m^2)$ operations}
			\EndFor
			\State $\hat{\tr} \leftarrow (m/2)^{-1} \sum_{i=1}^m \hat{\tr}_i$
			\State $\hat{\err}^2 \gets ((m/2)(m/2 - 1))^{-1} \sum_{i=1}^{m/2} (\hat{\tr}_i - \hat{\tr})^2$
		\end{algorithmic}
	\end{algorithm}
	
	The derivation of an $\order(m^2N)$ implementation for \NysTrace (shown in \cref{list:xnystrace}) is similar in spirit to the one just provided for \textsc{XTrace}\xspace and is omitted.
	
	\section{MATLAB implementations} \label{sec:implementation}
	
	We present optimized MATLAB R2022b implementations of \textsc{XTrace}\xspace, \NysTrace,\! and \XDiag in \cref{list:xtrace,list:xnystrace,list:xdiag}.
	
\begin{lstlisting}[float=t,frame=single, caption={MATLAB 2022b implementation for \textsc{XTrace}\xspace},label={list:xtrace}]
function [t,err] = xtrace(A, m)
N = size(A,1); m = floor(m/2);

%% Helper functions
cnormc = @(M) M ./ vecnorm(M,2,1);
diag_prod = @(A,B) sum(conj(A).*B,1).'; % computes diag(A'*B)

%% Randomized SVD
Om = sqrt(N)*cnormc(randn(N, m));
Y = A*Om;
[Q,R] = qr(Y,0);

%% Normalization
W = Q'*Om; 
S = cnormc(inv(R)');
scale = (N - m + 1) ./ (N - (vecnorm(W)').^2 ...
        + abs(diag_prod(S,W) .* vecnorm(S)').^2);

%% Quantities needed for trace estimation
Z = A*Q; H = Q'*Z; HW = H*W; T=Z'*Om;
dSW = diag_prod(S, W); dSHS = diag_prod(S, H*S);
dTW = diag_prod(T, W); dWHW = diag_prod(W, HW);
dSRmHW = diag_prod(S, R-HW); dTmHRS = diag_prod(T-H'*W,S);

%% Trace estimate
ests = trace(H)*ones(m,1) - dSHS + (dWHW - dTW + dTmHRS .* dSW...
    + abs(dSW).^2 .* dSHS + conj(dSW) .* dSRmHW) .* scale;
t = mean(ests);
err = std(ests)/sqrt(m);
end
\end{lstlisting}
	
\begin{lstlisting}[float=t,frame=single, caption={MATLAB 2022b implementation for \NysTrace},label={list:xnystrace}]
function [t,err] = xnystrace(A, m)
N = size(A,1);

%% Helper functions
cnormc = @(M) M ./ vecnorm(M,2,1);
diag_prod = @(A,B) sum(conj(A).*B,1).'; % computes diag(A'*B)

%% Nystrom
Om = sqrt(N) * cnormc(randn(N,m));
Y = A*Om;
nu = eps*norm(Y,'fro')/sqrt(N);
Y = Y + nu*Om; % Shift for numerical stability
[Q,R] = qr(Y,0);
H = Om'*Y; C = chol((H+H')/2);
B = R/C; % Nystrom approx is Q*B*B'*Q'

%% Normalization
[QQ,RR] = qr(Om,0);
WW = QQ'*Om; 
SS = cnormc(inv(RR)');
scale = (N - m + 1) ./ (N - vecnorm(WW).^2 ...
    + abs(diag_prod(SS,WW)' .* vecnorm(SS)).^2);

%% Trace estimate
W = Q'*Om; S = (B/C') .* (diag(inv(H))').^(-0.5);
dSW = diag_prod(S, W).';
ests = norm(B,'fro')^2 - vecnorm(S).^2 + abs(dSW).^2 .* scale...
    - nu*N;
t = mean(ests);
err = std(ests)/sqrt(m);
end
\end{lstlisting}
	
\begin{lstlisting}[float=t,frame=single, caption={MATLAB 2022b implementation for \XDiag},label={list:xdiag}]
function d = xdiag(A, m)
N = size(A,1); m = floor(m/2);

%% Helper functions
cnormc = @(M) M ./ vecnorm(M,2,1);
diag_prod = @(A,B) sum(conj(A).*B,1).'; % computes diag(A'*B)

%% Randomized SVD
Om = -3 + 2*randi(2,n,m); % Random signs
Y = A*Om; 
[Q,R] = qr(Y,0);

% Quantities needed for trace estimation
Z = A'*Q; T=Z'*Om;
S = cnormc(inv(R)');
dQZ = diag_prod(Q',Z'); dQSSZ = diag_prod((Q*S)',(Z*S)');
dOmQT = diag_prod(Om',(Q*T).'); dOmY = diag_prod(Om',Y.');
dOmQSST = diag_prod(Om',(Q*S*diag(diag_prod(S,T))).');

% Trace estimate
d = dQZ + (-dQSSZ+dOmY-dOmQT+dOmQSST)/m;
end
\end{lstlisting}
	
\end{document}